\pdfoutput=1
\documentclass[11pt]{article}
\usepackage{authblk}
\usepackage[toc,page]{appendix}
\usepackage[top=2cm, bottom=2cm, left=2cm, right=2cm]{geometry}

\usepackage{color}
\usepackage{helvet}         
\usepackage{courier}        
\usepackage{type1cm}        

\usepackage{framed}
\usepackage{tikz}
\usepackage{makeidx}         
\usepackage{graphicx}        
\usepackage{multicol}        
\usepackage[bottom]{footmisc}

\usepackage{amsmath}
\usepackage{amssymb}
\usepackage{bbold}
\usepackage{amsthm}
\usepackage{subcaption}
\usepackage{sidecap}
\usepackage{floatrow}
\usepackage{pdflscape}
\usepackage{pdflscape}
\usepackage{comment}
\usepackage[font=small]{caption}
\usepackage{enumitem}
\usepackage{esint}

\usepackage{scalerel}

\newtheorem{theorem}{Theorem}
\newtheorem{corollary}[theorem]{Corollary}
\newtheorem{lemma}[theorem]{Lemma}

\newtheorem{proposition*}{Proposition}
\newtheorem{lemma*}{Lemma}

\theoremstyle{remark}

\numberwithin{theorem}{section}
\numberwithin{question}{section}
\numberwithin{figure}{section}
\numberwithin{equation}{section}

\begin{document}

\title{Critical Ising model, Multiple $\SLE_\kappa\left(\frac{\kappa-6}{2},\frac{\kappa-6}{2}\right)$ and $\beta$-Jacobi Ensemble}
\bigskip{}
\author{Mingchang Liu\thanks{liumc\_prob@163.com}}
\affil{KTH Royal Institute of Technology, Sweden}
\date{}

%

%

\global\long\def\CR{\mathrm{CR}}
\global\long\def\ST{\mathrm{ST}}
\global\long\def\SF{\mathrm{SF}}
\global\long\def\cov{\mathrm{cov}}
\global\long\def\dist{\mathrm{dist}}
\global\long\def\SLE{\mathrm{SLE}}
\global\long\def\hSLE{\mathrm{hSLE}}
\global\long\def\CLE{\mathrm{CLE}}
\global\long\def\GFF{\mathrm{GFF}}
\global\long\def\inte{\mathrm{int}}
\global\long\def\ext{\mathrm{ext}}
\global\long\def\inrad{\mathrm{inrad}}
\global\long\def\outrad{\mathrm{outrad}}
\global\long\def\dimH{\mathrm{dim}}
\global\long\def\capa{\mathrm{cap}}
\global\long\def\diam{\mathrm{diam}}
\global\long\def\free{\mathrm{free}}
\global\long\def\hF{{}_2\mathrm{F}_1}
\global\long\def\ghF{{}_3\mathrm{F}_2}
\global\long\def\simple{\mathrm{simple}}
\global\long\def\even{\mathrm{even}}
\global\long\def\odd{\mathrm{odd}}
\global\long\def\st{\mathrm{ST}}
\global\long\def\usf{\mathrm{USF}}
\global\long\def\Leb{\mathrm{Leb}}
\global\long\def\LP{\mathrm{LP}}
\global\long\def\coulomb{\LH}
\global\long\def\coulombnew{\LG}
\global\long\def\kfunc{p}
\global\long\def\OO{\mathcal{O}}
\global\long\def\Dist{\mathrm{Dist}}
\global\long\def\ball{D}

\global\long\def\eps{\epsilon}
\global\long\def\ov{\overline}
\global\long\def\U{\mathbb{U}}
\global\long\def\T{\mathbb{T}}
\global\long\def\HH{\mathbb{H}}
\global\long\def\LA{\mathcal{A}}
\global\long\def\LB{\mathcal{B}}
\global\long\def\LC{\mathcal{C}}
\global\long\def\LD{\mathcal{D}}
\global\long\def\LF{\mathcal{F}}
\global\long\def\LK{\mathcal{K}}
\global\long\def\LE{\mathcal{E}}
\global\long\def\LG{\mathcal{G}}
\global\long\def\LI{\mathcal{I}}
\global\long\def\LJ{\mathcal{J}}
\global\long\def\LL{\mathcal{L}}
\global\long\def\LM{\mathcal{M}}
\global\long\def\LN{\mathcal{N}}
\global\long\def\LQ{\mathcal{Q}}
\global\long\def\LR{\mathcal{R}}
\global\long\def\LT{\mathcal{T}}
\global\long\def\LS{\mathcal{S}}
\global\long\def\LU{\mathcal{U}}
\global\long\def\LV{\mathcal{V}}
\global\long\def\LW{\mathcal{W}}
\global\long\def\LX{\mathcal{X}}
\global\long\def\LY{\mathcal{Y}}
\global\long\def\PartF{\mathcal{Z}}
\global\long\def\LH{\mathcal{H}}
\global\long\def\LJ{\mathcal{J}}
\global\long\def\R{\mathbb{R}}
\global\long\def\C{\mathbb{C}}
\global\long\def\N{\mathbb{N}}
\global\long\def\Z{\mathbb{Z}}
\global\long\def\E{\mathbb{E}}
\global\long\def\PP{\mathbb{P}}
\global\long\def\QQ{\mathbb{Q}}
\global\long\def\A{\mathbb{A}}
\global\long\def\one{\mathbb{1}}
\global\long\def\bn{\mathbf{n}}
\global\long\def\MR{MR}
\global\long\def\cond{\,|\,}
\global\long\def\la{\langle}
\global\long\def\ra{\rangle}
\global\long\def\tree{\Upsilon}
\global\long\def\prob{\mathbb{P}}
\global\long\def\hm{\mathrm{Hm}}
\global\long\def\cross{\mathrm{Cross}}

\global\long\def\sf{\mathrm{SF}}
\global\long\def\wr{\varrho}

\global\long\def\Im{\operatorname{Im}}
\global\long\def\Re{\operatorname{Re}}

\global\long\def\ud{\mathrm{d}}
\global\long\def\pder#1{\frac{\partial}{\partial#1}}
\global\long\def\pdder#1{\frac{\partial^{2}}{\partial#1^{2}}}
\global\long\def\der#1{\frac{\ud}{\ud#1}}

\global\long\def\bZnn{\mathbb{Z}_{\geq 0}}

\global\long\def\Vfunc{\LG}
\global\long\def\gfunc{g^{(\rr)}}
\global\long\def\hfunc{h^{(\rr)}}

\global\long\def\SimplexInt{\rho}
\global\long\def\CubeInt{\widetilde{\rho}}

\global\long\def\ii{\mathfrak{i}}
\global\long\def\rr{\mathfrak{r}}
\global\long\def\chamber{\mathfrak{X}}
\global\long\def\Wchamber{\mathfrak{W}}

\global\long\def\SimplexIntKappa8{\SimplexInt}

\global\long\def\nested{\boldsymbol{\underline{\Cap}}}
\global\long\def\unnested{\boldsymbol{\underline{\cap\cap}}}
\global\long\def\unnested{\boldsymbol{\underline{\cap\cap}}}

\global\long\def\acycle{\vartheta}
\global\long\def\bcycle{\tilde{\acycle}}
\global\long\def\Gloop{\Theta}

\global\long\def\metric{\mathrm{dist}}

\global\long\def\adj#1{\mathrm{adj}(#1)}

\global\long\def\bs{\boldsymbol}

\global\long\def\edge#1#2{\langle #1,#2 \rangle}
\global\long\def\graph{G}

\newcommand{\conn}{\vartheta}
\newcommand{\hatconn}{\widehat{\vartheta}_{\mathrm{RCM}}}
\newcommand{\realpt}{\smash{\mathring{x}}}
\newcommand{\corrind}{\LC}
\newcommand{\bssymb}{\pi}
\newcommand{\PRCM}{\mu}
\newcommand{\coeff}{p}
\newcommand{\MainConst}{C}

\global\long\def\removeLink{/}
\maketitle

\begin{center}
\begin{minipage}{0.95\textwidth}
\abstract{
Fix $N\ge 1$ and suppose that $(\Omega;x_1,\ldots, x_{N}; x_{N+1}, x_{N+2})$ is a (topological) polygon, i.e. $\Omega$ is a simply connected domain with locally connected boundary $\partial\Omega$ and $x_1,\ldots,x_{N+2}$ are $N+2$ different points located counterclockwisely on $\partial\Omega$. Fix $\kappa\in (0,4)$. In this paper, we will give two different constructions of multiple $N$-$\SLE_\kappa\left(\frac{\kappa-6}{2},\frac{\kappa-6}{2}\right)$ on $(\Omega;x_1,\ldots,x_{N}; x_{N+1},x_{N+2})$ and prove that they give the same law on random curves. Then, by establishing the uniqueness of multiple $N$-$\SLE_\kappa\left(\frac{\kappa-6}{2},\frac{\kappa-6}{2}\right)$, we can obtain the joint law of the hitting points of multiple $N$-$\SLE_\kappa\left(\frac{\kappa-6}{2},\frac{\kappa-6}{2}\right)$ with odd (resp. even) indices on $(x_{N+1}x_{N+2})$. After shrinking  $x_1,\ldots,x_N$ to one point, the law of hitting points with odd (resp. even) indices will converge to $\beta$-Jacobi ensemble with the conjectured relation $\beta=\frac{8}{\kappa}$. We will establish a direct connection between $\SLE$-type curves and $\beta$-Jacobi ensemble.

As an application, we consider critical Ising model on a discrete topological polygon $(\Omega^\delta_\delta;x^\delta_1,\ldots,x^\delta_{N}; x^\delta_{N+1},x^\delta_{N+2})$ on $\delta\Z^2$ with alternating boundary conditions on $(x^\delta_{N+2}x^\delta_{N+1})$ and free boundary condition on $(x^\delta_{N+1}x^\delta_{N+2})$. Motivated by the partition function of multiple $N$-$\SLE_\kappa\left(\frac{\kappa-6}{2},\frac{\kappa-6}{2}\right)$, we will derive the scaling limit of the probability of the event that the interface $\gamma_j^\delta$ starting from $x^\delta_j$ ends at $(x^\delta_{N+1}x^\delta_{N+2})$ for all $1\le j\le N$. Moreover, we will prove that given this event, the interface $(\gamma_1^\delta,\ldots,\gamma_N^\delta)$ converges to multiple $N$-$\SLE_\kappa\left(\frac{\kappa-6}{2},\frac{\kappa-6}{2}\right)$ with $\kappa=3$.
 }

\bigskip{}


\end{minipage}
\end{center}

\tableofcontents

\section{Introduction}
In 1999, Schramm introduced Schramm-Loewner evolution (SLE) as a candidate for the scaling limit of interfaces in two-dimensional lattice model at criticality. There are several models whose interfaces have been proved to converge to SLE process: loop-erased random walk and uniform spanning tree~\cite{LawlerSchrammWernerLERWUST}, percolation~\cite{SmirnovPercolationConformalInvariance}, level lines of Gaussian free field (GFF)~\cite{SchrammSheffieldDiscreteGFF}, Ising and FK-Ising model~\cite{ChelkakSmirnovIsing, CDCHKSConvergenceIsingSLE}. In these cases,  from the conformal invariance and domain Markov property, the boundary condition ensures that the driving function of random curve in the scaling limit is a Brownian motion $(B_{\kappa t}:t\ge 0)$ for some $\kappa>0$. When the boundary conditions are more complicated, the drift term of the driving function is non-zero and it will be related to certain partition functions. For instance, in~\cite{FengPeltolaWuConnectionProbaFKIsing,FengLiuPeltolaWu2024,LiuPeltolaWuUST}, the authors considered alternating boundary condition and the boundaries are allowed to be connected outside of the domain via planar partitions. 
In these cases, one can identify the conditional law of one interface given others in the discrete setting and thus the scaling limit should be multiple $\SLE_\kappa$: the conditional law of one curve given others equals the law of $\SLE_\kappa$ in the corresponding domain. 

In this paper, we will consider critical Ising model in the same setting as in~\cite{IzyurovObservableFree}, where both $+$ boundary, $-$ boundary and free boundary appear.  Fix $N\ge 1$ and suppose that $(\Omega;x_1,\ldots, x_{N}; x_{N+1}, x_{N+2})$ is a (topological) polygon, i.e. $\Omega$ is a simply connected domain with locally connected boundary $\partial\Omega$ and $x_1,\ldots,x_{N+2}$ are $N+2$ different boundary points located counterclockwisely on $\partial\Omega$. Denote by $(\Omega_\delta;x^\delta_1,\ldots,x^\delta_{N}; x^\delta_{N+1},x^\delta_{N+2})$ its discrete approximation on $\delta\Z^2$. We consider the critical Ising model on $\Omega_\delta$ with the following boundary conditions (we use the convention that $x^\delta_0:=x^\delta_{N+2}$):
\begin{equation}\label{eqn::boundary_data}
(-1)^j \text{ along }(x^\delta_{j-1}x^\delta_j)\quad\text{for }1\le j\le N+1,\quad\text{and free on }(x^\delta_{N+1}x^\delta_{N+2}).
\end{equation}
Under~\eqref{eqn::boundary_data}, there exist $N$ interfaces $\gamma^\delta_1,\ldots,\gamma^\delta_N$ such that $\gamma^\delta_i$ starts from $x^\delta_i$ for $1\le i\le N$. Each $\gamma^\delta_i$ will end at $x^\delta_j$ for some $j\neq i$ or at the free boundary $(x^\delta_{N+1}x^\delta_{N+2})$.  In~\cite{IzyurovObservableFree}, the author proved the convergence of  $\gamma^\delta_1$ stopped at a time before it ends, when the mesh of lattice tends to zero. (See~\cite[Equation 3.6 and equation 3.7]{FengWuYangIsing} for a more explicit formula of the driving function).
We focus on the event $A_\delta=A_\delta(\Omega_\delta;x_1^\delta,\ldots,x_N^\delta;x_{N+1}^\delta,x_{N+2}^\delta):=\{\gamma^\delta_i\text{ ends at }(x^\delta_{N+1}x^\delta_{N+2}),\text{ for }1\le i\le N\}$. Note that $A_\delta$ is equivalent to the event that for $1\le j\le N+1$, there exists a cluster with sign $(-1)^j$ connecting $(x^\delta_{j-1}x^\delta_j)$ and $(x^\delta_{N+1}x^\delta_{N+2})$. See Figure~\ref{fig::Ising} for an illustration.
\begin{figure}[ht!]
\begin{center}
\includegraphics[width=0.8\textwidth]{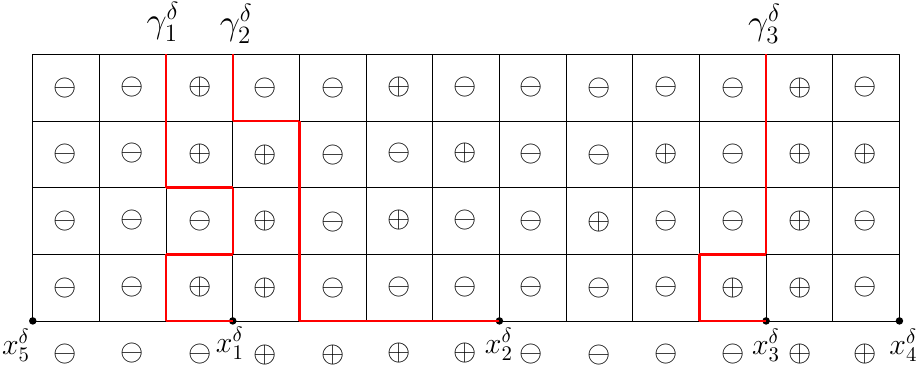}
\end{center}
\caption{\label{fig::Ising} This is an illustration of the setup of Ising model in this paper when $N=3$. Along $(x_5^\delta x_4^\delta)$, the boundary conditions are alternating boundary conditions; along $(x_4^\delta x_5^\delta)$, the boundary condition is free boundary condition. In this picture, the event $A_\delta$ occurs and the discrete curves in red are interfaces starting from $x_1^\delta,x_2^\delta$ and $x_3^\delta$.}
\end{figure}
We will derive the scaling limit of $\PP[A_\delta]$. Moreover,  we will prove that given $A_\delta$, the conditional law of $\{\gamma^\delta_1,\ldots,\gamma^\delta_N\}$ converges to the multiple $N$-$\SLE_\kappa\left(\frac{\kappa-6}{2},\frac{\kappa-6}{2}\right)$ with $\kappa=\frac{8}{3}$. See Theorem~\ref{thm::Ising_limit}.

Fix $\kappa\in (0,4)$. In Theorem~\ref{thm::multiple_odd} and Theorem~\ref{thm::multiple_even}, we will give two different constructions of multiple $N$-$\SLE_\kappa\left(\frac{\kappa-6}{2},\frac{\kappa-6}{2}\right)$. In Theorem~\ref{thm::uniqueness_multiple}, we will prove that multiple $N$-$\SLE_\kappa\left(\frac{\kappa-6}{2},\frac{\kappa-6}{2}\right)$ is characterised by its conditional law. In particular,  we can obtain the joint law of hitting points of multiple $N$-$\SLE_\kappa\left(\frac{\kappa-6}{2},\frac{\kappa-6}{2}\right)$ on $(x_{N+1}x_{N+2})$ with odd (resp. even) indices, see Corollary~\ref{coro::hitting_law}.

By the explicit form of the law of hitting points~\eqref{eqn::density_case1} and~\eqref{eqn::density_case2}, intuitively, by shrinking $x_1,\ldots,x_N$ to one point, the law of the hitting points with odd (resp. even) indices equals the $\beta$-Jacobi ensemble with the conjectured relation $\beta=\frac{8}{\kappa}$ from physics literature~\cite{CardyDyson}. But it is not obvious that how to prove the convergence of the law of multiple $N$-$\SLE_\kappa\left(\frac{\kappa-6}{2},\frac{\kappa-6}{2}\right)$ when shrinking the starting points to one point. Thus, we will construct another type of multiple $N$-$\SLE_\kappa\left(\frac{\kappa-6}{2},\frac{\kappa-6}{2}\right)$ such that all the curves start from the same point. Similarly to the constructions given in Theorem~\ref{thm::multiple_odd} and Theorem~\ref{thm::multiple_even},  we will give two constructions in Theorem~\ref{thm::same_point_odd} and Theorem~\ref{thm::same_point_even}. By establishing the uniqueness similar to Theorem~\ref{thm::uniqueness_multiple} (see Theorem~\ref{thm::same_point_uniqueness}), we obtain a direct connection between $\SLE$-type curves and $\beta$-Jacobi ensemble with $\beta=\frac{8}{\kappa}$, see Corollary~\ref{coro::same_point_hitting_law}.

\subsection{Multiple $N$-$\SLE_\kappa\left(\frac{\kappa-6}{2},\frac{\kappa-6}{2}\right)$}
\label{subsec::intro_multiple}
We denote by $X$ the set of unparameterised planar oriented curves, i.e. continuous mappings from $[0,1]$ to $\C$ modulo reparameterization. We equip $X$ with the metric 
\begin{equation}\label{eqn::curve_metric}
\dist(\eta, \tilde{\eta}):=\inf_{\varphi,\tilde{\varphi}}\sup_{t\in [0,1]}|\eta(\varphi(t))-\tilde{\eta}(\tilde{\varphi}(t))|,
\end{equation}
where the infimum is taken over all increasing homeomorphisms $\varphi, \tilde{\varphi}: [0,1]\to [0,1]$. Then, the metric space $(X, \dist)$ is complete and separable.

Fix $N\ge 1$ and fix a polygon $(\Omega;x_1,\ldots,x_N; x_{N+1},x_{N+2})$. Denote by $X_N(\Omega;x_1,\ldots,x_N; x_{N+1},x_{N+2})$ the set of curves     $(\gamma_1,\ldots,\gamma_N)$ in $X^N$ such that $\gamma_j$ is a simple curve which starts from $x_j$ and ends at $(x_{N+1},x_{N+2})$ for $1\le j\le N$, and $\gamma_i\cap\gamma_j=\emptyset$ for $1\le i<j\le N$. For  $\kappa\in(0,4)$, we call a probability measure on $(\gamma_1,\ldots,\gamma_N)\in X(\Omega;x_1,\ldots,x_N;x_{N+1},x_{N+2})$ a multiple $N$-$\SLE_\kappa\left(\frac{\kappa-6}{2},\frac{\kappa-6}{2}\right)$ if the conditional law of one curve given others equals $\SLE_\kappa\left(\frac{\kappa-6}{2},\frac{\kappa-6}{2}\right)$. More precisely,
for $1\le j\le N$, define $\tau_j$ to be the hitting time of $\gamma_j$ on $(x_{N+1}x_{N+2})$.  We use the convention $\gamma_0(\tau_0):=x_{N+2}$ and $\gamma_{N+1}(\tau_{N+1}):=x_{N+1}$. Denote by $\Omega_j$ the connected component of $\Omega\setminus\cup_{i\neq j}\gamma_i$ whose boundary contains $\gamma_{j-1}(\tau_{j-1})$ and $\gamma_{j+1}(\tau_{j+1})$. Given $\cup_{i\neq j}\gamma_i$, the conditional law of $\gamma_j$ equals the law of $\SLE_\kappa\left(\frac{\kappa-6}{2},\frac{\kappa-6}{2}\right)$ curve on $\Omega_j$ starting from $x_j$ with two force points $\gamma_{j-1}(\tau_{j-1})$ and $\gamma_{j+1}(\tau_{j+1})$. 

We will give two different constructions of multiple $N$-$\SLE_\kappa\left(\frac{\kappa-6}{2},\frac{\kappa-6}{2}\right)$ via flow lines of Gaussian free field~\cite{MillerSheffieldIG1} (see also in Section~\ref{sec::flow_line} for a brief introduction). Define $\lambda:=\frac{\pi}{\sqrt\kappa}$ and define $\chi:=\frac{2}{\sqrt\kappa}-\frac{\sqrt\kappa}{2}$. Denote by $M:=\left[\frac{N}{2}\right]$ and by $\bold x:=(x_1,\ldots,x_N)$. 
\begin{theorem}\label{thm::multiple_odd}
Fix $N\ge 1$ and $\kappa\in (0,4)$, and fix $x_1<\cdots<x_N<x_{N+1}$. We construct a probability measure on $(\ell_1,\ldots,\ell_N)\in X_N(\HH;x_1,\ldots,x_{N};x_{N+1},\infty)$ as follows.
\begin{itemize}
\item
Sample $\bold z:=(z_1,\ldots,z_M)$ with density (with respect to the product of Lebesgue measure)
\begin{align}\label{eqn::density_case1}
&r(N; \bold x,x_{N+1};\bold z)\notag\\
=&\frac{1}{\mathcal{Z}_N}\one_{\{x_{N+1}<z_1<\cdots<z_M\}}\prod_{\substack{1\le i\le N\\1\le j\le M}}(z_j-x_i)^{-\frac{4}{\kappa}}\prod_{1\le j\le M}(z_j-x_{N+1})^{\frac{6-\kappa}{\kappa}}\prod_{1\le i<j\le M}(z_j-z_i)^{\frac{8}{\kappa}},
\end{align}
where the normalisation $\mathcal{Z}_N=\mathcal{Z}_N(\bold x; x_{N+1})$ is defined by 
\begin{align*}
&\mathcal{Z}_N\\
:=&\int_{\{x_{N+1}<z_1<\cdots<z_M\}}\prod_{\substack{1\le i\le N\\1\le j\le M}}(z_j-x_i)^{-\frac{4}{\kappa}}\prod_{1\le j\le M}(z_j-x_{N+1})^{\frac{6-\kappa}{\kappa}}\prod_{1\le i<j\le M}(z_j-z_i)^{\frac{8}{\kappa}}\prod_{1\le j\le M}\ud z_j.
\end{align*}
\item
Given $(z_1,\ldots,z_M)$, suppose $\Gamma$ is the Gaussian free field with boundary data
\begin{align*}
&\lambda-2N\lambda\text{ on }(-\infty,x_1),\quad\lambda\left(1-2N+2(j-1)\right)\text{ on }(x_{j-1},x_j)\text{ for }2\le j\le N,\\
& \lambda\text{ on }(x_N,x_{N+1}),\quad \lambda\left(1+\frac{\kappa-6}{2}\right)\text{ on }(x_{N+1},z_1),\\
&\lambda\left(1+\frac{\kappa-6}{2}-4(j-1)\right)\text{ on }(z_{j-1},z_j)\text{ for }2\le j\le M,\,
\lambda\left(1+\frac{\kappa-6}{2}-4M\right)\text{ on }(z_{M},+\infty).
\end{align*}
For $1\le j\le N$, suppose $\ell_j$ is the flow line of $\Gamma$ starting from $x_j$ with angle $\frac{2\lambda}{\chi}(j-1)$. 
\end{itemize}
Then, the law of random curves $(\ell_1,\ldots,\ell_N)$ is a multiple $N$-$\SLE_\kappa\left(\frac{\kappa-6}{2},\frac{\kappa-6}{2}\right)$. See Figure~\ref{fig::ell} for an illustration.
\end{theorem}
\begin{figure}[ht!]
\begin{center}
\includegraphics[width=0.8\textwidth]{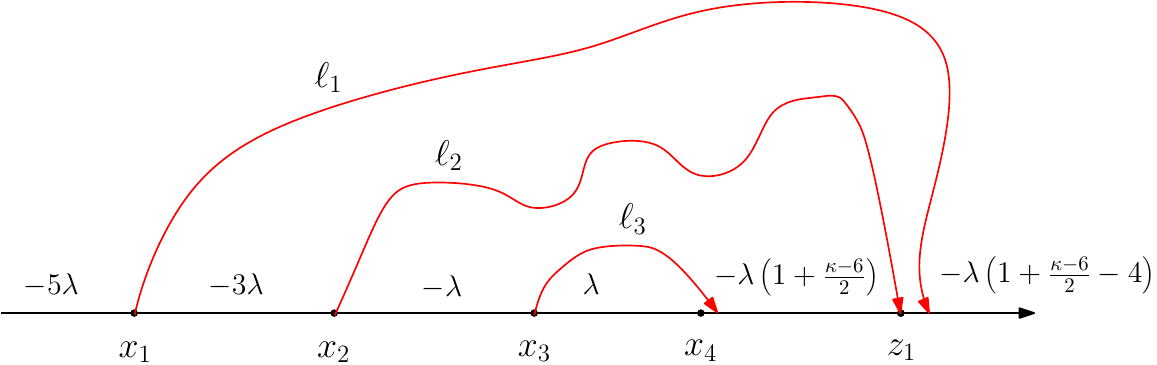}
\end{center}
\caption{\label{fig::ell} This is an illustration of the Theorem~\ref{thm::multiple_odd} when $N=3$. Note that $\ell_2$ hits $(x_4,+\infty)$ at $z_1$.}
\end{figure}
\begin{theorem}\label{thm::multiple_even}
Fix $N\ge 1$ and $\kappa\in (0,4)$, and fix $x_1<\cdots<x_N<x_{N+1}$. We construct a probability measure on $(\eta_1,\ldots,\eta_N)\in X_N(\HH;x_1,\ldots,x_{N};x_{N+1},\infty)$ as follows.
\begin{itemize}
\item
Sample $\bold w:=(w_1,\ldots,w_{N-M})$ with density (with respect to the Lebesgue measure)
\begin{align}\label{eqn::density_case2}
&\rho(N;\bold x, x_{N+1};\bold w)\notag\\
=&\frac{1}{\mathcal{W}_N}\one_{\{x_{N+1}<w_1<\cdots<w_{N-M}\}}\prod_{\substack{1\le i\le N\\1\le j\le N-M}}(w_j-x_i)^{-\frac{4}{\kappa}}\prod_{1\le j\le N-M}(w_j-x_{N+1})^{\frac{2-\kappa}{\kappa}}\prod_{1\le i<j\le N-M}(w_j-w_i)^{\frac{8}{\kappa}},
\end{align}
where the normalisation $\mathcal{W}_N=\mathcal{W}_N(\bold x; x_{N+1})$ is defined by
\begin{align*}
&\mathcal{W}_N\\
:=&\int_{\{x_{N+1}<w_1<\cdots<w_{N-M}\}}\prod_{\substack{1\le i\le N\\1\le j\le N-M}}(w_j-x_i)^{-\frac{4}{\kappa}}\prod_{1\le j\le N-M}(w_j-x_{N+1})^{\frac{2-\kappa}{\kappa}}\prod_{1\le i<j\le N-M}(w_j-w_i)^{\frac{8}{\kappa}}\prod_{1\le j\le N-M}\ud w_j.
\end{align*}
\item
Given $(w_1,\ldots,w_{N-M})$, suppose $\hat\Gamma$ is the Gaussian free field with boundary data
\begin{align*}
&\lambda-2N\lambda\text{ on }(-\infty,x_1),\quad\lambda\left(1-2N+2(j-1)\right)\text{ on }(x_{j-1},x_j)\text{ for }2\le j\le N,\\
& \lambda\text{ on }(x_N,x_{N+1}),\quad \lambda\left(1+\frac{\kappa-2}{2}\right)\text{ on }(x_{N+1},w_1),\\
&\lambda\left(1+\frac{\kappa-2}{2}-4(j-1)\right)\text{ on }(w_{j-1},w_j)\text{ for }2\le j\le N-M,\,
\lambda\left(1+\frac{\kappa-2}{2}-4(N-M)\right)\text{ on }(w_{N-M},+\infty).
\end{align*}
For $1\le j\le N$, suppose $\eta_j$ is the flow line of $\hat\Gamma$ starting from $x_j$ with angle $\frac{2\lambda}{\chi}(j-1)$. 
\end{itemize}
Then, the law of random curves $(\eta_1,\ldots,\eta_N)$ is a multiple $N$-$\SLE_\kappa\left(\frac{\kappa-6}{2},\frac{\kappa-6}{2}\right)$. See Figure~\ref{fig::eta} for an illustration.
\end{theorem}
\begin{figure}[ht!]
\begin{center}
\includegraphics[width=0.8\textwidth]{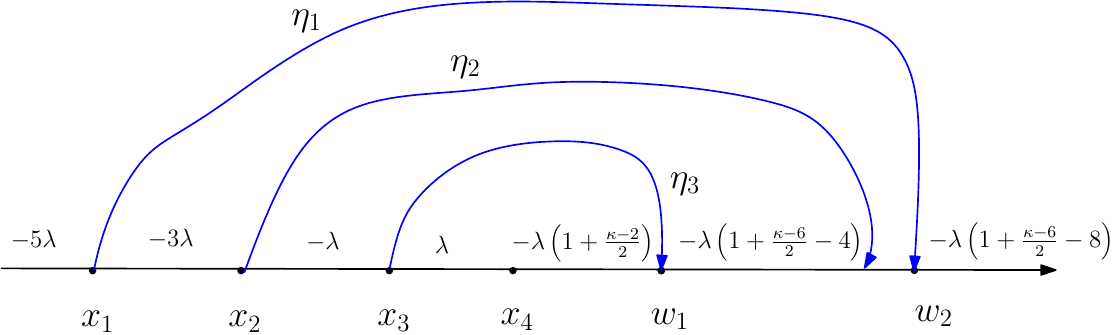}
\end{center}
\caption{\label{fig::eta} This is an illustration of the Theorem~\ref{thm::multiple_even} when $N=3$. Note that $\eta_1$ hits $(x_4,+\infty)$ at $w_1$ and $\eta_3$ hits $(x_4,+\infty)$ at $w_2$.}
\end{figure}
Similarly to the case of multiple $\SLE_\kappa$ (for instance, see~\cite{BeffaraPeltolaWuUniqueness,ZhanExistenceUniquenessMultipleSLE}), multiple $\SLE_\kappa\left(\frac{\kappa-6}{2},\frac{\kappa-6}{2}\right)$ is also characterised by its conditional law. 
\begin{theorem}\label{thm::uniqueness_multiple}
Fix $N\ge 1$ and $\kappa\in(0,4)$. Fix a  polygon $(\Omega;x_1,\ldots,x_N;x_{N+1},x_{N+2})$. The law of multiple $N$-$\SLE_\kappa\left(\frac{\kappa-6}{2},\frac{\kappa-6}{2}\right)$ on $X_N(\Omega;x_1,\ldots,x_N;x_{N+1},x_{N+2})$ is unique.
\end{theorem}
By Theorem~\ref{thm::uniqueness_multiple}, we can obtain the information of hitting points from both Theorem~\ref{thm::multiple_odd} and Theorem~\ref{thm::multiple_even}.
\begin{corollary}\label{coro::hitting_law}
Fix $N\ge 1$ and $\kappa\in(0,4)$. Sample multiple $N$-$\SLE_\kappa\left(\frac{\kappa-6}{2},\frac{\kappa-6}{2}\right)$ on $(\HH;x_1,\ldots,x_{N};x_{N+1},\infty)$, which we denote by $(\gamma_1,\ldots,\gamma_N)$. Recall that we denote by $\tau_j$ the hitting time of $(x_{N+1},+\infty)$ for $1\le j\le N$. Then, the law of hitting points with odd indices $(\gamma_{N-1}(\tau_{N-1}),\gamma_{N-3}(\tau_{N-3}),\ldots,\gamma_{N-2M+1}(\tau_{N-2M+1}))$ equals the law of $(z_1,\ldots, z_M)$ given in Theorem~\ref{thm::multiple_odd}. Similarly, the law of hitting points with even indices $(\gamma_N(\tau_N),\gamma_{N-2}(\tau_{N-2}),\ldots,\gamma_{N-(2N-2M-2)}(\tau_{N-(2N-2M-2)}))$ equals the law of $(w_1,\ldots, w_{N-M})$ given in Theorem~\ref{thm::multiple_even}.
\end{corollary}
It is clear that when shrinking $x_1,\ldots,x_N$ to one point, the law of $(w_1,\ldots,w_{N-M})$ (resp. $(z_1,\ldots, z_M)$) converges to Jacobi$\left(N-M;\frac{8}{\kappa},\frac{4M-2N+3}{2},\frac{1}{2}\right)$ (resp. Jacobi$\left(M;\frac{8}{\kappa},\frac{2N-4M+1}{2},\frac{3}{2}\right)$) (see Section~\ref{sec::Jacobi}).  But it is not obvious that how to prove the convergence of curves: additional singularity at the starting point appears when shrinking $x_1,\ldots,x_N$ to one point. Next, we will consider another type of multiple $\SLE_\kappa\left(\frac{\kappa-6}{2},\frac{\kappa-6}{2}\right)$ and we will prove that  the hitting points with odd (resp. even) indices forms Jacobi$\left(N-M;\frac{8}{\kappa},\frac{4M-2N+3}{2},\frac{1}{2}\right)$ (resp. Jacobi$\left(M;\frac{8}{\kappa},\frac{2N-4M+1}{2},\frac{3}{2}\right)$).

Fix $N\ge 1$ and fix a  polygon $(\Omega;x;u,v)$. Define $X_N(\Omega;x;u,v)$ to be the set of curves   $(\gamma_1,\ldots,\gamma_N)$ in $X^N$ such that $\gamma_j$ is a simple curve which starts from $x$ and ends at $(uv)$ for $1\le j\le N$, and $\gamma_i\cap\gamma_j=\{x\}$ for $1\le i<j\le N$. For  $\kappa\in(0,4)$, we call a probability measure on $(\gamma_1,\ldots,\gamma_N)\in X(\Omega;x;u,v)$ is a multiple $N$-$\SLE_\kappa\left(\frac{\kappa-6}{2},\frac{\kappa-6}{2}\right)$ starting from $x$ if the conditional law of one curve given others equals $\SLE_\kappa\left(\frac{\kappa-6}{2},\frac{\kappa-6}{2}\right)$. More precisely,
for $1\le j\le N$, define $\tau_j$ to be the hitting time of $\gamma_j$ on $(uv)$.  We use the convention $\gamma_0(\tau_0):=u$ and $\gamma_{N+1}(\tau_{N+1}):=v$. Denote by $\Omega_j$ the connected component of $\Omega\setminus\cup_{i\neq j}\gamma_i$ whose boundary contains $\gamma_{j-1}(\tau_{j-1})$ and $\gamma_{j+1}(\tau_{j+1})$. Given $\cup_{i\neq j}\gamma_i$, the conditional law of $\gamma_j$ equals the law of $\SLE_\kappa\left(\frac{\kappa-6}{2},\frac{\kappa-6}{2}\right)$ on $\Omega_j$ curve starting from $x$ with two force points $\gamma_{j-1}(\tau_{j-1})$ and $\gamma_{j+1}(\tau_{j+1})$. 

Similarly to Theorem~\ref{thm::multiple_odd} and Theorem~\ref{thm::multiple_even}, we will give two different constructions of multiple $N$-$\SLE_\kappa\left(\frac{\kappa-6}{2},\frac{\kappa-6}{2}\right)$ starting from $x$. 
\begin{theorem}\label{thm::same_point_odd}
Fix $N\ge 1$ and $\kappa\in (0,4)$, and fix $x<u$. We construct a probability measure on $(\ell_1,\ldots,\ell_N)\in X_N(\HH;x;u,\infty)$ as follows.
\begin{itemize}
\item
Sample $\bold z:=(z_1,\ldots,z_M)$ with density (with respect to the Lebesgue measure)
\begin{align}\label{eqn::density_case1}
&r(N; x,u;\bold z)
=\frac{1}{\mathcal{Z}_N}\one_{\{u<z_1<\cdots<z_M\}}\prod_{1\le j\le M}(z_j-x)^{-\frac{4N}{\kappa}}\prod_{1\le j\le M}(z_j-u)^{\frac{6-\kappa}{\kappa}}\prod_{1\le i<j\le M}(z_j-z_i)^{\frac{8}{\kappa}},
\end{align}
where the normalisation $\mathcal{Z}_N=\mathcal{Z}_N(x; u)$ is defined by 
\begin{align*}
&\mathcal{Z}_N
:=\int_{\{u<z_1<\cdots<z_M\}}\prod_{1\le j\le M}(z_j-x)^{-\frac{4N}{\kappa}}\prod_{1\le j\le M}(z_j-u)^{\frac{6-\kappa}{\kappa}}\prod_{1\le i<j\le M}(z_j-z_i)^{\frac{8}{\kappa}}\prod_{1\le j\le M}\ud z_j.
\end{align*}
\item
Given $(z_1,\ldots,z_M)$, suppose $\Gamma$ is the Gaussian free field with boundary data
\begin{align*}
&\lambda-2N\lambda\text{ on }(-\infty,x_1),\quad \lambda\text{ on }(x,u),\quad \lambda\left(1+\frac{\kappa-6}{2}\right)\text{ on }(u,z_1),\\
&\lambda\left(1+\frac{\kappa-6}{2}-4(j-1)\right)\text{ on }(z_{j-1},z_j)\text{ for }2\le j\le M,\,
\lambda\left(1+\frac{\kappa-6}{2}-4M\right)\text{ on }(z_{M},+\infty).
\end{align*}
For $1\le j\le N$, suppose $\ell_j$ is the flow line of $\Gamma$ starting from $x$ with angle $\frac{2\lambda}{\chi}(j-1)$. 
\end{itemize}
Then, the law of random curves $(\ell_1,\ldots,\ell_N)$ is a multiple $N$-$\SLE_\kappa\left(\frac{\kappa-6}{2},\frac{\kappa-6}{2}\right)$ starting from $x$. See Figure~\ref{fig::ell1} for an illustration.
\end{theorem}
\begin{figure}[ht!]
\begin{center}
\includegraphics[width=0.8\textwidth]{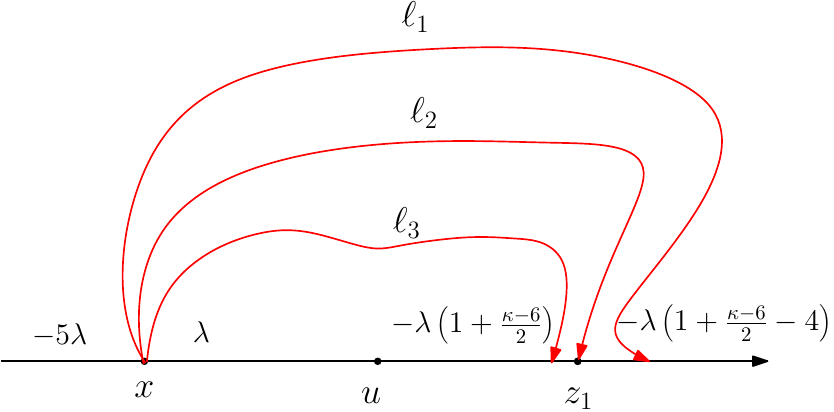}
\end{center}
\caption{\label{fig::ell1} This is an illustration of the Theorem~\ref{thm::multiple_odd} when $N=3$. Note that $\ell_2$ hits $(x_4,+\infty)$ at $z_1$.}
\end{figure}
\begin{theorem}\label{thm::same_point_even}
Fix $N\ge 1$ and $\kappa\in (0,4)$, and fix $x<u$. We construct a probability measure on $(\eta_1,\ldots,\eta_N)\in X_N(\HH;x;u ,\infty)$ as follows.
\begin{itemize}
\item
Sample $\bold w:=(w_1,\ldots,w_{N-M})$ with density (with respect to the Lebesgue measure)
\begin{align}\label{eqn::density_case2}
&\rho(N;x,u;\bold w)\notag\\
=&\frac{1}{\mathcal{W}_N}\one_{\{u<w_1<\cdots<w_{N-M}\}}\prod_{1\le j\le N-M}(w_j-x)^{-\frac{4N}{\kappa}}\prod_{1\le j\le N-M}(w_j-u)^{\frac{2-\kappa}{\kappa}}\prod_{1\le i<j\le N-M}(w_j-w_i)^{\frac{8}{\kappa}},
\end{align}
where the normalisation $\mathcal{W}_N=\mathcal{W}_N(x; u)$ is defined by
\begin{align*}
&\mathcal{W}_N\\
:=&\int_{\{u<w_1<\cdots<w_{N-M}\}}\prod_{1\le j\le N-M}(w_j-x)^{-\frac{4N}{\kappa}}\prod_{1\le j\le N-M}(w_j-u)^{\frac{2-\kappa}{\kappa}}\prod_{1\le i<j\le N-M}(w_j-w_i)^{\frac{8}{\kappa}}\prod_{1\le j\le N-M}\ud w_j.
\end{align*}
\item
Given $(w_1,\ldots,w_{N-M})$, suppose $\hat\Gamma$ is the Gaussian free field with boundary data
\begin{align*}
&\lambda-2N\lambda\text{ on }(-\infty,x),\quad \lambda\text{ on }(x, u),\quad \lambda\left(1+\frac{\kappa-2}{2}\right)\text{ on }(u,w_1),\\
&\lambda\left(1+\frac{\kappa-2}{2}-4(j-1)\right)\text{ on }(w_{j-1},w_j)\text{ for }2\le j\le N-M,\,
\lambda\left(1+\frac{\kappa-2}{2}-4(N-M)\right)\text{ on }(w_{N-M},+\infty).
\end{align*}
For $1\le j\le N$, suppose $\eta_j$ is the flow line of $\hat\Gamma$ starting from $x$ with angle $\frac{2\lambda}{\chi}(j-1)$. 
\end{itemize}
Then, the law of random curves $(\eta_1,\ldots,\eta_N)$ is a multiple $N$-$\SLE_\kappa\left(\frac{\kappa-6}{2},\frac{\kappa-6}{2}\right)$ starting from $x$. See Figure~\ref{fig::eta1} for an illustration.
\end{theorem}
\begin{figure}[ht!]
\begin{center}
\includegraphics[width=0.8\textwidth]{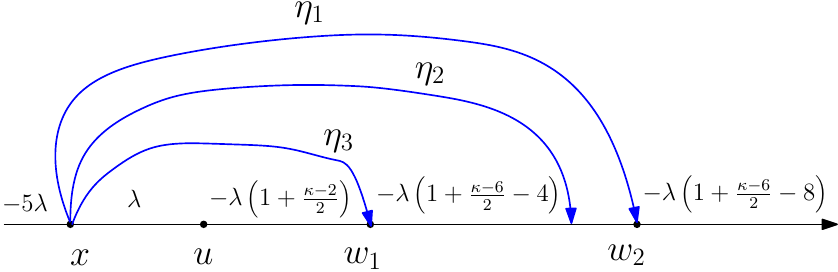}
\end{center}
\caption{\label{fig::eta1} This is an illustration of the Theorem~\ref{thm::multiple_odd} when $N=3$. Note that $\eta_1$ hits $(x_4,+\infty)$ at $w_1$ and $\eta_3$ hits $(x_4,+\infty)$ at $w_2$.}
\end{figure}
Based on Theorem~\ref{thm::uniqueness_multiple}, we can obtain the uniqueness of multiple $\SLE_\kappa$ starting from the same point.
\begin{theorem}\label{thm::same_point_uniqueness}
Fix $N\ge 1$ and $\kappa\in(0,4)$. Fix a polygon $(\Omega;x;u,v)$. The law of multiple $N$-$\SLE_\kappa\left(\frac{\kappa-6}{2},\frac{\kappa-6}{2}\right)$ starting from $x$ on $X_N(\Omega;x;u,v)$ is unique.
\end{theorem}
Define $\psi(z):=\frac{1}{z}$ for $z\in\R$.  Combining Theorem~\ref{thm::same_point_odd}--Theorem~\ref{thm::same_point_uniqueness} together, we can obtain two $\beta$-Jacobi ensembles from one multiple $N$-$\SLE_\kappa\left(\frac{\kappa-6}{2},\frac{\kappa-6}{2}\right)$ starting from the same point. Moreover, under this coupling, these two $\beta$-Jacobi  ensembles interlace.
\begin{corollary}\label{coro::same_point_hitting_law}
Fix $N\ge 1$ and $\kappa\in(0,4)$. Sample $N$-multiple $\SLE_\kappa$ on $(\HH;0;1,\infty)$, which we denote by $(\gamma_1,\ldots,\gamma_N)$. Denote by $\tau_j$ the hitting time of $\gamma_j$ at $(1,+\infty)$ for $1\le j\le N$. Then, the law of hitting points with odd indices $(\psi(\gamma_{N-1}(\tau_{N-1})),\psi(\gamma_{N-3}(\tau_{N-3})),\ldots,\psi(\gamma_{N-2M+1}(\tau_{N-2M+1})))$ equals Jacobi$\left(M;\frac{8}{\kappa},\frac{2N-4M+1}{2},\frac{3}{2}\right)$. Similarly, the law of hitting points with even indices 
\[(\psi(\gamma_N(\tau_N)),\psi(\gamma_{N-2}(\tau_{N-2})),\ldots,\psi(\gamma_{N-(2N-2M-2)}(\tau_{N-(2N-2M-2)})))\]
 equals Jacobi$\left(N-M;\frac{8}{\kappa},\frac{4M-2N+3}{2},\frac{1}{2}\right)$.
\end{corollary}
In~\cite{HealeyLawlerNSidedRadialSLE}, the authors considered the $N$-sided radial $\SLE_\kappa$ for $\kappa\in (0,4)$. In that setting, the driving function is given by circular Dyson's Brownian motion. But it is not clear that how to recover $\beta$-ensemble itself from $\SLE$-type curves. In this paper, we answer this question in a special setting. Moreover, from the connection between flow lines and GFF, our results indicate that eigenvalues of standard matrix models should appear in GFF naturally. This is related the central limit theorem of $\beta$-ensebmble.
\subsection{Critical Ising model}
Fix a topological polygon $(\Omega;x_1,\ldots,x_N;x_{N+1},x_{N+2})$. Suppose $(\Omega_\delta;x_1^\delta,\ldots,x_N^\delta;x_{N+1}^\delta,x_{N+2}^\delta)$ is a discrete topological polygon on $\delta\Z^2$, such that $(\Omega_\delta;x_1^\delta,\ldots,x_N^\delta;x_{N+1}^\delta,x_{N+2}^\delta)$ converges to $(\Omega;x_1,\ldots,x_N;x_{N+1},x_{N+2})$ in the close-Carath\'eodory sense (see~\cite[Theorem 4.2]{KarrilaConformalImage}) : 
\begin{itemize}
\item
There exists conformal maps $\varphi_\delta$ from the unit disk $\U$ onto $\Omega_\delta$ for every $\delta>0$ and conformal map $\varphi$ from $\U$ onto $\Omega$, such that $\varphi_\delta$ converges to $\varphi$ locally uniformly. Moreover, we require that $x_j^\delta$ converges to $x_j$ and $\varphi^{-1}_\delta(x_j^\delta)$ converges to $\varphi^{-1}(x_j)$.
\item
Given a reference point $u\in\Omega$ and $r>0$ small enough, let $S_r$ be the connected component of $\partial B(x_j,r)\cap\Omega$ disconnecting $x_j$ from $u$ and from the boundary arc $(x_{j+2}x_{j+1})$. We require that, for
each  $r$ small enough and for all sufficiently small $\delta$ (depending on $r$),  the boundary point $x_j^\delta$ connects to the midpoint of $S_r$ inside $\Omega_\delta\cap B(x_j,r)$.
\end{itemize}
Define $\hat\Omega_\delta$ to be the graph obtained by adding all the squares adjacent to the boundary arc $(x_{N+2}^\delta x_{N+1}^\delta)$.  The critical Ising model on $\Omega_\delta$ with boundary conditions~\eqref{eqn::boundary_data}  is a random assignment $\sigma$ from faces of $\hat\Omega_\delta$ to $\{\pm 1\}$, such that $\sigma$ equals $(-1)^j$ on the  squares in $\hat\Omega_\delta\setminus\Omega_\delta$  adjacent to $(x_j^\delta x_{j+1}^\delta)$ for $0\le j\le N$ (we use the convention $x_0^\delta=x_{N+2}^\delta$) and the law of $\sigma$ is given by
\[\PP[\sigma=\sigma_0]=\frac{1}{Z}e^{\beta_c\sum_{x,y\in\hat\Omega_\delta: x\sim y}\sigma_0(x)\sigma_0(y)},\]
where $\beta_c=-\frac{1}{2}\log(\sqrt 2-1)$ and the normalisation $Z$ is defined by 
\[Z=\sum_{\sigma}e^{\beta_c\sum_{x,y\in\hat\Omega_\delta: x\sim y}\sigma_0(x)\sigma_0(y)}.\]
For $1\le j\le N$, denote by $\gamma^\delta_j$ the interface starts from $x^\delta_j$: it is a discrete path on $\delta\Z^2$ which starts from $x_\delta$ and it continues such that $(-1)^j$ spin is on the left and $(-1)^{j+1}$ spin is on the right. It turns left when there exist two possible ways to continue. Recall that we denote by $A_\delta=\{\gamma^\delta_i\text{ ends at }(x^\delta_{N+1} x^\delta_{N+2}),\text{ for }1\le i\le N\}$.

Recall the definitio of partition function $\LR_N :\Xi_{N+1}=\{(y_1,\ldots,y_{N+1}):y_1<\ldots<y_{N+1}\}\to\R$ in~\cite[Equation 3.6--Equation 3.8]{FengWuYangIsing}. For $n\ge 1$, denote by $\Pi_n$ the set of pair partitions $\omega=\{\{a_1,b_1\},\ldots,\{a_n,b_n\}\}$ of $\{1,\ldots,2n\}$, such that $a_1<\ldots<a_n$ and $a_j<b_j$ for $1\le j\le n$. Define $\text{sgn }(\omega)$ to be the sign of the product $\prod_{1\le i<j\le n}(a_i-a_j)(a_i-b_j)(b_i-a_j)(b_i-b_j)$.
\begin{itemize}
\item
When $N=2m$ for some $m\ge 1$, define
\[R_N(y_1,\ldots,y_{2m};y_{2m+1}):=\prod_{1\le k\le 2m}\frac{1}{\sqrt{y_{2m+1}-y_k}}\times\sum_{\omega\in\Pi_m}\text{ sgn }(\omega)\prod_{\{a,b\}\in\omega}\frac{2y_{2m+1}-y_{a}-y_{b}}{y_{b}-y_{a}}.\]
\item
When $N=2m+1$ for some $m\ge 0$, define
\[R_N(y_1,\ldots,y_{2m+1};y_{2m+2}):=\prod_{1\le k\le 2m+1}\frac{1}{\sqrt{y_{2m+2}-y_k}}\times\sum_{\omega\in\Pi_{m+1}}\text{ sgn }(\omega)\prod_{\substack{\{a,b\}\in\omega\\b\neq 2m+2}}\frac{2y_{2m+2}-y_{a}-y_{b}}{y_{b}-y_{a}}.\]
\end{itemize}
\begin{theorem}\label{thm::Ising_limit}
Suppose $(\Omega_\delta;x_1^\delta,\ldots,x_N^\delta;x_{N+1}^\delta,x_{N+2}^\delta)$ converges to $(\Omega;x_1,\ldots,x_N;x_{N+1},x_{N+2})$ in the close-Carath\'eodory sense. Fix a conformal map $\phi$ from $\Omega$ onto $\HH$ such that $\phi(x_1)<\ldots<\phi(x_{N+1})$ and $\phi(x_{2N+2})=\infty$. Then, we have ($\kappa=3$)
\begin{align}\label{eqn::limt_prob}
&\lim_{\delta\to 0}\PP[A_\delta]\notag\\
=&\frac{1}{B\left(\frac{2}{\kappa},\frac{2}{\kappa}\right)^{\left[\frac{N}{2}\right]}}\prod_{1\le i<j\le N}(\phi(x_j)-\phi(x_i))^{\frac{2}{\kappa}}\prod_{1\le i\le N}(\phi(x_{N+1})-\phi(x_i))^{\frac{\kappa-6}{2\kappa}}\times\frac{\mathcal Z_N(\phi(x_1),\ldots,\phi(x_N);\phi(x_{N+1}))}{\LR_N(\phi(x_1),\ldots,\phi(x_N);\phi(x_{N+1}))}\notag\\
=&\frac{1}{B\left(\frac{2}{\kappa},\frac{2}{\kappa}\right)^{\left[\frac{N+1}{2}\right]}}\prod_{1\le i<j\le N}(\phi(x_j)-\phi(x_i))^{\frac{2}{\kappa}}\prod_{1\le i\le N}(\phi(x_{N+1})-\phi(x_i))^{\frac{\kappa-2}{2\kappa}}\times\frac{\mathcal W_N(\phi(x_1),\ldots,\phi(x_N);\phi(x_{N+1}))}{\LR_N(\phi(x_1),\ldots,\phi(x_N);\phi(x_{N+1}))},
\end{align}
where $B(\cdot,\cdot)$ is the Beta function. Moreover, given $A_\delta$, the conditional law of interfaces $(\gamma_1^\delta,\ldots,\gamma_N^\delta)$ converges to multiple  $N$-$\SLE_\kappa\left(\frac{\kappa-6}{2},\frac{\kappa-6}{2}\right)$ on $X_N(\Omega;x_1,\ldots,x_N;x_{N+1},x_{N+2})$, under the metric~\eqref{eqn::curve_metric}.
\end{theorem}



\section{Preliminaries}
\label{sec::pre}
\subsection{Loewner chains and SLE}
\label{sec::SLE}
A compact subset $K$ of $\overline{\HH}$ is an $\HH$-hull if $\HH\setminus K$ is simply connected. By Riemann's mapping theorem, there exists a unique conformal map $g_K$ from $\HH\setminus K$ onto $\HH$ normalized at $\infty$: $\lim_{z\to\infty}|g_K(z)-z|=0$. A Loewner chain is a sequence of growing $\HH$-hulls determined by the Loewner equation: for $z\in\overline{\HH}$, 
\[\partial_t g_t(z)=\frac{2}{g_t(z)-W_t},\quad g_0(z)=z, \]
where $(W_t: t\ge 0)$ is a real-valued continuous function, which we call driving function. Define $\tau(z):=\sup\{t\ge 0: \min_{s\in [0,t]}|g_s(z)-W_s|>0\}$ and define $K_t:=\{z\in\HH: \tau(z)\le t\}$. Then $g_t$ is the unique conformal map from $\HH\setminus K_t$ onto $\HH$ normalized at $\infty$. The collection of $\HH$-hulls $(K_t: t\ge 0)$ is called a Loewner chain. 


For $\kappa\ge0$, the $\SLE_{\kappa}$ process in $\HH$ from $0$
to $\infty$ is a random Loewner chain driving by $W_t=\sqrt{\kappa}B_t$ where $(B_t: t\ge 0)$ is a standard one-dimensional Brownian motion. 
Rohde and Schramm~\cite{RohdeSchrammSLEBasicProperty} proved that there exists a curve $\gamma$ such that $\HH\setminus K_t$ is the same as the unbounded connected component of $\HH\setminus\gamma[0,t]$. We call this curve $\SLE_{\kappa}$ curve in $\HH$ from $0$ to $\infty$. Such curve is simple when $\kappa\in (0,4]$. In general simply connected domains, the $\SLE_{\kappa}$ curve is defined via conformal invariance.
\subsection{GFF flow lines}
\label{sec::flow_line}
We will give a very brief summary on Gaussian free field (GFF) and its flow lines, more detailed information can be found in~\cite{MillerSheffieldIG1}. Suppose $D\subset\C$ is a simply connected domain.  The zero-boundary GFF on $D$ is a random distribution, such that for every smooth function $f$ with compact support on $D$, we have that $(h,f)$ is a Gaussian variable with
\[\E[(h,f)]=0\quad\text{ and }\quad \E[(h,f)^2]=\int_{D}\int_D f(x)G_D(x,y)f(y)dx dy,\]
where $G_D(\cdot,\cdot)$ is the Green function on $D$ with Dirichlet boundary condition, and the normalisation is given by that $G_D(x,y)+\log|x-y|$ is bounded when $x\to y$.
The GFF with boundary data $\Gamma_0$ is the sum of the zero-boundary GFF on $D$ and the harmonic extension of $\Gamma_0$ on $D$. 

Now we consider coupling between GFF in $\HH$ and SLE processes.
We fix $\lambda=\pi/\sqrt{\kappa}$. 
 Fix $x_1<\cdots<x_n$ and $\rho_1, \ldots, \rho_{n+1}\in\R$ and fix $i\in \{1, \ldots, n\}$.  On the one hand, we define 
\[\overline{\rho}_j=\sum_{k=j}^{i}\rho_k,\text{ for }j\in\{1, \ldots, i\}; \qquad \overline{\rho}_j=\sum_{k=i+1}^j\rho_k,\text{ for }j\in\{i+1, \ldots, n+1\},\]
and consider GFF $\Gamma$ in $\HH$ with the boundary data: (with convention $x_0=-\infty$ and $x_{n+1}=\infty$)
\begin{align*}
-\lambda(1+\overline{\rho}_j) \text{ on }(x_{j-1}, x_j),\text{ for }j\in \{1, \ldots, i\}\, ;\, \lambda(1+\overline{\rho}_{j+1})\text{ on }(x_j, x_{j+1}), \text{ for }j\in \{i, \ldots, n\}. 
\end{align*}
On the other hand, we consider $\gamma\sim\SLE_{\kappa}(\rho_1, \ldots, \rho_{i}; \rho_{i+1}, \ldots, \rho_{n+1})$ in $\HH$ starting from $x_i$ with force points $x_1^L=x_i^-,\ldots,x_{i}^L=x_1$ on its left and $x_1^R=x_i^+,\ldots,x_{n-i+1}^R=x_n$. The driving function $(W_t:t\ge 0)$ of $\gamma$ satisfies the SDEs
\begin{align}\label{eqn::rhoSDE}
\begin{cases}
d W_t=\sqrt{\kappa}d B_t + \sum_{1\le j\le i}\frac{\rho_j dt}{W_t-V_t^{j,L}}+\sum_{1\le j\le n-i+1}\frac{\rho_{j+1} dt}{W_t-V_t^{j,R}},\quad W_0=x_i; \\
d V_t^{j,L} = \frac{2d t}{V_t^{j,L}-W_t},\quad V_0^{j,L}=x_j^L \text{ for }1\le j\le i;\quad d V_t^{j,R} = \frac{2d t}{V_t^{j,R}-W_t},\quad V_0^{j,R}=x_j^R \text{ for }1\le j\le n-i+1,
\end{cases}
\end{align}
where $(B_t:t\ge 0)$ is the standard one-dimensional  Brownian motion starting from $0$.
It is shown in~\cite{MillerSheffieldIG1} that there exists a coupling $(\Gamma, \gamma)$ such that $\gamma$ is the ``flow line" of $\Gamma$. In particular, the law of flow line is determined by the boundary data. The flow line with angle $\theta$ is the flow line of $\Gamma+\theta\chi$. In general simply connected domains, the coupling between GFF and flow lines is defined via conformal invariance.
The theory of how the flow lines of the GFF interact with each other is developed in~\cite{MillerSheffieldIG1}. 
We list the properties used in this paper below.

Denote by $\PP_N=\PP_N(x_1,\ldots,x_{N};x_{N+1})$ the probability measure corresponding to the coupling given in Theorem~\ref{thm::multiple_odd}.  We have the following properties.
\begin{itemize}
\item
By~\cite{MillerSheffieldIG1}, for $1\le j\le N$,  the curve $\ell_j$ has the same law as $\SLE_\kappa\left(2,\ldots,2; 2,\ldots,2,\frac{\kappa-6}{2},-4,\ldots,-4\right)$, with force points $x_1^L=x_{j-1},\ldots, x_{j-1}^L=x_1$ and $x_1^R=x_{j+1},\ldots, x^R_{N-j}=x_{N}$, $x_{N-j+1}^R=x_{N+1}$, $x_{N-j+1+m}^R=z_m$ for $1\le m\le M$. Under the coupling of flow lines, almost surely, $\ell_{m-1}$ is on the left to $\ell_{m}$ for $2\le m\le N$.
\item
Denote by $z_0:=x_{N+1}$ and by $z_{M+1}:=\infty$. By the same proof as~\cite[Lemma 15]{DubedatSLEDuality},  for $j\ge 0$ with $2j+1\le N$, we have that $\ell_{N-(2j+1)}$ ends at $z_{j+1}$ almost surely. Moreover, for $1\le j\le N$, the random curve $\ell_j$ hits $\R$ only at its two ends almost surely.
\item
For $j\ge 0$ and $2j\le N$, given $\cup_{i\neq N-2j}\ell_i$, if we denote by $\Omega_{N-2j}$ the connected component of $\HH\setminus \cup_{i\neq N-2j}\ell_{i}$ which contains $x_{N-2j}$, the conditional law of $\ell_{N-2j}$ equals $\SLE_\kappa\left(\frac{\kappa-6}{2},\frac{\kappa-6}{2}\right)$ starting from $x_{N-2j}$ with force points $z_j$ on its right and $z_{j+1}$ on its left. By~\cite[Lemma 4.3]{MillerSheffieldIG1}, for $j\ge 0$ with $2j\le N$, we have that $\ell_{N-2j}$ ends at $(z_j,z_{j+1})$ almost surely.
\end{itemize}

Denote by $\QQ_N=\QQ_N(x_1,\ldots,x_N;x_{N+1})$ the probability measure corresponding to the coupling given in Theorem~\ref{thm::multiple_even}. We have the following properties.
\begin{itemize}
\item
By~\cite{MillerSheffieldIG1}, for $1\le j\le N$, the curve $\eta_j$ has the same law as $\SLE_\kappa\left(2,\ldots,2; 2,\ldots,2,\frac{\kappa-2}{2},-4,\ldots,-4\right)$, with force points $x_1^L=x_{j-1},\ldots, x_{j-1}^L=x_1$ and $x_1^R=x_{j+1},\ldots, x^R_{N-j}=x_{N}$, $x_{N-j+1}^R=x_{N+1}$, $x_{N-j+1+m}^R=w_m$ for $1\le m\le N-M$. Under the coupling of flow lines, almost surely, $\eta_{m-1}$ is on the left to $\eta_{m}$ for $2\le\ell\le N$.
\item
Denote by  $w_{N-M+1}:=\infty$. By the same proof as~\cite[Lemma 15]{DubedatSLEDuality}, for $j\ge 0$ with $2j\le N$, we have that $\eta_{N-2j}$ ends at $w_{j+1}$ almost surely. Moreover, for $1\le j\le N$, the random curve $\eta_j$ hits $\R$ only at its two ends almost surely.
\item
For $j\ge 0$ and $2j+1\le N$, given $\cup_{i\neq N-(2j+1)}\eta_i$, if we denote by $\Omega_{N-(2j+1)}$ the connected component of $\HH\setminus \cup_{i\neq N-(2j+1)}\eta_{N-i}$ which contains $x_{N-(2j+1)}$, the conditional law of $\eta_{N-(2j+1)}$ equals $\SLE_\kappa\left(\frac{\kappa-6}{2},\frac{\kappa-6}{2}\right)$ starting from $x_{N-(2j+1)}$ with force points $w_{j+1}$ on its right and $w_{j+2}$ on its left. By~\cite[Lemma 4.3]{MillerSheffieldIG1}, for $j\ge 0$ with $2j+1\le N$, we have that $\eta_{N-(2j+1)}$ ends at $(w_{j+1},w_{j+2})$ almost surely.
\end{itemize}

\subsection{$\beta$-Jacobi ensemble}
\label{sec::Jacobi}
Fix $N\ge 1$, $\beta>0$ and $a,b>-1$. The $\beta$-Jacobi ensemble with parameters $N,\beta,a,b$ describes $N$ real particles $(y_1,\ldots, y_N)$ distributed according to the joint probability density function (with respect to Lebesgue measure)
\[\rho(y_1,\ldots,y_N)=\frac{1}{Z_{J}}\one_{\{0<y_1<\ldots<y_N<1\}}\prod_{1\le i<j\le N}|y_j-y_i|^{\beta}\prod_{1\le j\le N}x_j^a(1-x_j)^b,\]
where the normalisation $Z_J$ is defined by 
\[Z_J=\int_{\{0<y_1<\ldots<y_N<1\}}\prod_{1\le i<j\le N}|y_j-y_i|^{\beta}\prod_{1\le j\le N}y_j^a(1-y_j)^b \prod_{1\le j\le N}\ud y_j.\]
In~\cite{RImatrix}, the authors introduced a type of tri-diagonal matrix model whose eigenvalues are distributed as Jacobi$(N;\beta;a,b)$. 

\subsection{Ising model}
\label{sec::Ising}
In this section, we recall some facts of Ising model.  Suppose $D$ is a discrete domain on $\Z^2$, denote by $\sigma\sim\PP_D^{\xi}$ the  critical Ising model on $D$ with boundary condition $\xi$. It satisfies the following two basic properties.
\begin{itemize} 
\item
Domain Markov property: Suppose $\hat D\subset D$ is a subdomain, then we have
\[\PP_D^{\xi}[\cdot\cond \sigma |_{D\setminus\hat D}]=\PP_{\hat D}^{\xi\cup\sigma |_{D\setminus\hat D}}[\cdot].\]  
\item
Monotonicity with respect boundary conditions: Suppose $A$ is an increasing event, i.e. if $\sigma\in A$ and $\sigma\le\sigma'$, we have that $\sigma'\in A$. Suppose $\xi\le\xi'$,  then we have
\[\PP_D^{\xi}[A]\le \PP_D^{\xi'}[A].\]
\end{itemize}
In this paper, we will use the following strong RSW estimate. Suppose $(D;a,b,c,d)$ is a discrete topological rectangle on $\Z^2$, denote by $d_D((ab),(cd))$ the discrete extremal distance between $(ab)$ and $(cd)$. By~\cite[Proposition 6.2]{ChelkakRobustComplexAnalysis}, the discrete extremal distance is uniformly comparable with respect to the extremal distance in the continuous setting. 
\begin{lemma}\label{lem::RSW}\cite[Corollary 1.7]{CDCHcrossing}
There exists $\eta=\eta(L)>0$, such that for every discrete topological rectangle $(D;a,b,c,d)$ with $d_D((ab),(cd))\le L$, we have that
\[\PP_D^{\xi}[\text{ there is a crossing of }-1 \text{ spins connecting }(ab)\text{ and }(cd)]\ge \eta,\]
where we denote by $\xi$ the following boundary condition:  it is free along $(ab)$ and $(cd)$ and it equals $+1$ along $(bc)$ and $(da)$.
\end{lemma}
\subsection{Notations}
\label{sec::not}
For $z\in\C$ and $r>0$, denote by $B(z,r):=\{w\in\C: |w-z|=r\}$.

Suppose $\Omega$ is a simply connected domain on $\overline\HH$ and $(-\infty,x)\subset\partial\Omega$. Denote by $x^-$ the prime end determined by the curve $(-\infty,x)$. For $x,y\in\partial\Omega$, we denote by $(xy)$ the arc on $\partial\Omega$ from $x$ to $y$ counterclockweisely.

\section{Proof of Theorem~\ref{thm::multiple_odd}--Corollary~\ref{coro::hitting_law}}
\label{sec::general_ensemble}
In this section, we will prove Theorem~\ref{thm::multiple_odd}--Corollary~\ref{coro::hitting_law}.

Fix $N\ge 1$ and fix $x_1<\ldots<x_{N+1}$. We first derive cascade relations for $(\ell_1,\ldots,\ell_N)\sim\PP_N$ and $(\eta_1,\ldots,\eta_N)\sim\QQ_N$ in Lemma~\ref{lem::cascade_eta} and Lemma~\ref{lem::cascade_gamma}. For a simple curve $\ell$ starting from $x_N$, denote by $(W^\ell_t:t\ge 0)$ the driving function and by $(g^\ell_t:t\ge 0)$ the corresponding conformal maps. Denote by $\tau$ the hitting time of curves at $(x_{N+1},+\infty)$. For every $\eps>0$, denote by $\tau_\eps$ the hitting time of the union of $\partial B\left(x_N,\frac{1}{\eps}\right)$ and the $\eps$-neighbourhood of $(x_{N+1},+\infty)$.
\begin{lemma}\label{lem::cascade_eta}
Denote by $\ell$ the $\SLE_\kappa\left(2,\ldots,2;\frac{\kappa-6}{2}\right)$ curve from $x_N$ to $\infty$, with force points given by $x_1^L=x_{N-1},\ldots, x_{N-1}^L=x_1$ and $x_1^R=x_{N+1}$. 
Then, we have the following properties.
\begin{itemize}
\item
The process $\{\mathcal Z_N(g^\ell_{t}(x_1),\ldots, g^\ell_{t}(x_{N-1}), W^\ell_{t}; g^\ell_{t}(x_{N+1}))\}_{t\ge 0}$ is a local martingale for $\ell$. Moreover, for every $\eps>0$, the law of $\ell_N[0,\tau_\eps]$ equals the law of $\ell[0,\tau_\eps]$ weighted by
\begin{equation}\label{eqn::weight_elleps}
\frac{\mathcal Z_N(g^\ell_{\tau_\eps}(x_1),\ldots, g^\ell_{\tau_\eps}(x_{N-1}), W^\ell_{\tau_\eps}; g^\ell_{\tau_\eps}(x_{N+1}))}{\mathcal{Z}_N(\bold x; x_{N+1})},
\end{equation}
\item
The law of $\ell_N[0,\tau]$ equals the law of $\ell[0,\tau]$ weighted by
\begin{equation}\label{eqn::weight_ell}
\frac{\mathcal{W}_{N-1}(g^\ell_{\tau}(x_1),\ldots,g^\ell_{\tau}(x_{N-1}); W^\ell_{\tau})}{\mathcal{Z}_N(\bold x; x_{N+1})},
\end{equation}
where we use the convention that $\mathcal{W}_0(\cdot; \cdot)=\mathcal{Z}_1(\cdot; \cdot)=1$.
\item
Given $\ell_N[0,\tau]$, the conditional law of $(g^{\ell_N}_{\tau}(\ell_1),\ldots, g^{\ell_N}_{\tau}(\ell_{N-1}))$ equals the law of random curves under $\QQ_{N-1}(g^\ell_{\tau}(x_1),\ldots,g^\ell_{\tau}(x_{N-1}); W^\ell_{\tau})$.
\end{itemize}
\end{lemma}
\begin{proof}
We only need to consider the case that $N\ge 2$. Define
\begin{align*}
U_t:=&\prod_{\substack{1\le i\le N-1\\1\le j\le M}}(g^\ell_t(z_j)-g^{\ell}_t(x_i))^{-\frac{4}{\kappa}}\prod_{1\le j\le M}(g_t^\ell(z_j)-W^\ell_t)^{-\frac{4}{\kappa}}\prod_{1\le j\le M}(g^\ell_t(z_j)-g^\ell_t(x_{N+1}))^{\frac{6-\kappa}{\kappa}}\\
&\times\prod_{1\le i<j\le M}(g_t^\ell(z_j)-g_t^\ell(z_i))^{\frac{8}{\kappa}}\prod_{1\le j\le M}\left(g_t^\ell\right)'(z_j).
\end{align*}
Note that there exists $C=C(\eps,z_1,\ldots,z_M)$ such that
\[0\le U_{t\wedge\tau_\eps}\le C.\]
Combining with~\cite[Theorem 6]{SchrammWilsonSLECoordinatechanges}, we have that $\{U_{t\wedge\tau_\eps}\}_{t\ge 0}$ is a martingale for $\ell$. Moreover, the law of $\ell[0,\tau_\eps]$ weighted by $\left\{\frac{U_{t\wedge\tau_\eps}}{U_0}\right\}_{t\ge 0}$ equals the conditional law of $\ell_N[0,\tau_\eps]$ given $z_1,\ldots,z_M$. 
Thus, for every bounded and continuous function $g$ on the curve space, and bounded and continuous function $f$, we have
\begin{align}
&\E[g(\ell_N[0,\tau_\eps])f(g^{\ell_N}_{\tau_\eps}(z_1),\ldots, g^{\ell_N}_{\tau_\eps}(z_M))]\notag\\
=&\E[\E[g(\ell_N[0,\tau_\eps])f(g^{\ell_N}_{\tau_\eps}(z_1),\ldots, g^{\ell_N}_{\tau_\eps}(z_M))\cond z_1,\ldots,z_M]]\notag\\
=&\int_{\{x_{N+1}<z_1\cdots<z_M\}}\E\left[g(\ell[0,\tau_\eps])f(g^{\ell}_{\tau_\eps}(z_1),\ldots, g^{\ell}_{\tau_\eps}(z_M))\frac{U_{\tau_\eps}}{U_0}\right]r(N;\bold x,x_{N+1};\bold z)\prod_{1\le j\le M}\ud z_j\notag\\
=&\frac{1}{\mathcal{Z}_N(\bold x;x_{N+1})}\E\left[g(\ell[0,\tau_\eps])U(f,N;g^\ell_{\tau_\eps}(x_1),\ldots, g^\ell_{\tau_\eps}(x_{N-1}), W^\ell_{\tau_\eps},g^\ell_{\tau_\eps}(u))\right],\label{eqn::truncated_mart_eta}
\end{align}
where we define 
\begin{align}\label{eqn::mart_eta}
&U(f,N;g^\ell_{t}(x_1),\ldots, g^\ell_{t}(x_{N-1}), W^\ell_{t},g^\ell_{t}(x_{N+1}))\notag\\
:=&\int_{\{g^\ell_{t}(x_{N+1})<u_1\cdots<u_M\}}\prod_{1\le i<j\le M}(u_j-u_i)^{\frac{8}{\kappa}}\prod_{\substack{1\le i\le N-1\\1\le j\le M}}(u_j-g^{\ell}_t(x_i))^{-\frac{4}{\kappa}}\prod_{1\le j\le M}(u_j-W^\ell_t)^{-\frac{4}{\kappa}}(u_j-g^\ell_t(x_{N+1}))^{\frac{6-\kappa}{\kappa}}\ud u_j\notag\\
&\times f(u_1,\ldots,u_M).
\end{align}
When $f=1$, by definition, we have that 
\[U(f,N;g^\ell_{t}(x_1),\ldots, g^\ell_{t}(x_{N-1}), W^\ell_{t},g^\ell_{t}(x_{N+1}))=\mathcal Z_N(g^\ell_{t}(x_1),\ldots, g^\ell_{t}(x_{N-1}), W^\ell_{t},g^\ell_{t}(x_{N+1})).\]
This gives the first item.

Next, we will prove that $\{U(f,N;g^\ell_{t}(x_1),\ldots, g^\ell_{t}(x_{N-1}), W^\ell_{t},g^\ell_{t}(x_{N+1}))\}_{0\le t\le \tau}$ is uniformly bounded and derive the terminal value. For $1\le i\le N-1$ and $1\le j\le M$, we have
\[u_j-W^\ell_{t}\ge u_j-g^\ell_{t}(x_{N+1})\quad\text{and}\quad u_j-g^{\ell}_t(x_i)\ge u_j-g^{\ell}_t(x_{N-1}).\]
Then, we have
\begin{align}\label{eqn::boud_intd}
&\prod_{\substack{1\le i\le N-1\\1\le j\le M}}(u_j-g^{\ell}_t(x_i))^{-\frac{4}{\kappa}}\prod_{1\le j\le M}(u_j-W^\ell_t)^{-\frac{4}{\kappa}}\notag\\
\le &\prod_{1\le j\le M}(u_j-g^{\ell}_t(x_{N-1}))^{-\frac{4}{\kappa}(N-1)}\prod_{1\le j\le M}(u_j-g^\ell_t(x_{N+1}))^{-\frac{4}{\kappa}}.
\end{align}
Note that~\eqref{eqn::boud_intd} implies that
\begin{align}\label{eqn::bound_U_1}
&\mathcal Z_N(g^\ell_{t}(x_1),\ldots, g^\ell_{t}(x_{N-1}), W^\ell_{t},g^\ell_{t}(x_{N+1}))\notag\\
\le&\int_{\{g^\ell_{t}(x_{N+1})<u_1\cdots<u_M\}}\prod_{1\le i<j\le M}(u_j-u_i)^{\frac{8}{\kappa}}\prod_{1\le j\le M}(u_j-g^{\ell}_{\tau_\eps}(x_{N-1}))^{-\frac{4}{\kappa}(N-1)}\prod_{1\le j\le M}(u_j-g^\ell_{\tau_\eps}(x_{N+1}))^{\frac{2-\kappa}{\kappa}}\ud u_j\notag\\
= &(g^\ell_{\tau_\eps}(x_{N+1})-g^\ell_{\tau_\eps}(x_{N-1}))^{\frac{4}{\kappa}M(M-N+\frac{1}{2})}\int_{0<u_1<\cdots<u_M<1}\prod_{1\le i<j\le M}(u_j-u_i)^{\frac{8}{\kappa}}\prod_{1\le j\le M}u_j^{\frac{4N-8M+2-\kappa}{\kappa}}(1-u_j)^{\frac{2-\kappa}{\kappa}}\ud u_j.
\end{align}
Note that
\[M-N+\frac{1}{2}<0\quad\text{and}\quad g^\ell_{\tau_\eps}(x_{N+1})-g^\ell_{\tau_\eps}(x_{N-1})\ge x_{N+1}-x_{N-1}.\]
Thus, by~\eqref{eqn::bound_U_1}, there exists a constant $C_3>0$ which is independent of $\eps$,  such that
\begin{equation}\label{eqn::bound_U_2}
U(f,N;g^\ell_{\tau_\eps}(x_1),\ldots,g^\ell_{\tau_\eps}(x_{N-1}),W^\ell_{\tau_\eps},g^\ell_{\tau_\eps}(u))\le C_3\max f.
\end{equation}
Moreover, by~\eqref{eqn::boud_intd},~\eqref{eqn::bound_U_1} and dominated convergence theorem,  we have that
\begin{align}\label{eqn::limit_U}
&\lim_{\eps\to 0}U(f,N;g^\ell_{\tau_\eps}(x_1),\ldots,g^\ell_{\tau_\eps}(x_{N-1}),W^\ell_{\tau_\eps},g^\ell_{\tau_\eps}(u))\notag\\
=&\mathcal{W}_{N-1}(g^\ell_{\tau}(x_1),\ldots,g^\ell_{\tau}(x_{N-1}); W^\ell_{\tau})\notag\\
&\times\int_{\{W^\ell_{\tau}<u_1<\cdots<u_{M}\}}f(u_1,\ldots,u_{M})\rho(N-1;g^\ell_{\tau}(x_1),\ldots,g^\ell_{\tau}(x_{N-1}),W^\ell_{\tau};\bold u)\prod_{1\le j\le M}\ud u_j\notag\\
:=&U(f,N;g^\ell_{\tau}(x_1),\ldots,g^\ell_{\tau}(x_{N-1}),W^\ell_{\tau}).
\end{align}
Combining with~\eqref{eqn::truncated_mart_eta} and~\eqref{eqn::limit_U},  letting $\eps\to 0$, by dominated convergence theorem, we have
\begin{align}
&\E[g(\ell_N[0,\tau])f(g^{\ell_N}_{\tau}(z_1),\ldots, g^{\ell_N}_{\tau}(z_M))]=\frac{1}{\mathcal{Z}_N(\bold x; x_{N+1})}\E\left[g(\ell[0,\tau])U(f,N;g^\ell_{\tau}(x_1),\ldots,g^\ell_{\tau}(x_{N-1}),W^\ell_{\tau})\right].\label{eqn::mart_eta}
\end{align}
Taking $f=1$, we obtain~\eqref{eqn::weight_ell}.  This gives the second item.

Note that~\eqref{eqn::mart_eta}  implies that the density of the conditional law of $(g^{\ell_N}_{\tau}(z_1),\ldots, g^{\ell_N}_{\tau}(z_M))$ given $\ell_N[0,\tau]$ is given by 
$\rho(N-1;g^{\ell_N}_{\tau}(x_1),\ldots, g^{\ell_N}_{\tau}(x_{N-1}), W^{\ell_1}_{\tau};\cdot,\ldots,\cdot)$. Combining with the property of coupling of flow lines, we have that the conditional law of $(g^{\ell_N}_{\tau}(\ell_1),\ldots, g^{\ell_N}_{\tau}(\ell_{N-1}))$ equals the law of random curves under $\QQ_{N-1}(g^{\ell_N}_{\tau}(x_1),\ldots,g^{\ell_N}_{\tau}(x_{N-1}); W^{\ell_N}_{\tau})$. This gives the third item and completes the proof.

\end{proof}
\begin{lemma}\label{lem::cascade_gamma}
Denote by $\ell$ the $\SLE_\kappa\left(2,\ldots,2;\frac{\kappa-6}{2}\right)$ curve from $x_N$ to $\infty$, with force points given by $x_1^L=x_{N-1},\ldots, x_{N-1}^L=x_1$ and $x_1^R=x_{N+1}$.  
Then, we have the following properties.
\begin{itemize}
\item
Define
\[S_t:=\prod_{1\le i\le N-1}(g_t^\ell(x_{N+1})-g_t^\ell(x_i))^{\frac{2}{\kappa}}\times(g_t^\ell(x_{N+1})-W_t^\ell)^{\frac{2}{\kappa}}\times \mathcal W_N(g^\ell_{t}(x_1),\ldots,g^\ell_{t}(x_{N-1}),W^\ell_{t}; g^\ell_{t}(x_{N+1})).\]
Then, the process $\{S_t\}_{t\ge 0}$ is a local martingale for $\ell$. Moreover, for every $\eps>0$, the law of $\eta_N[0,\tau_\eps]$ equals the law of $\ell[0,\tau_\eps]$ weighted by $\frac{S_{\tau_\eps}}{S_0}$.
\item
The law of $\eta_N[0,\tau]$ equals the law of $\ell[0,\tau]$ weighted by
\begin{equation}\label{eqn::weight_eta}
B\left(\frac{2}{\kappa},\frac{2}{\kappa}\right)\times\frac{\prod_{1\le i\le N-1}(g_\tau^\ell(x_{N+1})-g_\tau^\ell(x_i))^{-\frac{2}{\kappa}}\mathcal{Z}_{N-1}(g_\tau^\ell(x_1),\ldots,g_\tau^\ell(x_{N-1}); W^\ell_{\tau})}{\prod_{1\le i\le N-1}(x_{N+1}-x_i)^{\frac{2}{\kappa}}\mathcal{W}_N(\bold x; x_{N+1})},
\end{equation}
where $B(\cdot,\cdot)$ is the Beta function and we use the convention that $\mathcal{Z}_0(\cdot;\cdot)=\mathcal{Z}_1(\cdot;\cdot)=1$.

\item
Given $\eta_N$, the conditional law of $(g^{\eta_N}_{\tau}(\eta_1),\ldots, g^{\eta_N}_{\tau}(\eta_{N-1}))$ equals the law of random curves under $\PP_{N-1}(g^{\eta_N}_{\tau}(x_1),\ldots,g^{\eta_N}(x_{N-1}); W^{\eta_N}_{\tau})$.
\end{itemize}
\end{lemma}
\begin{proof}
When $N=1$, note that 
\begin{equation}\label{eqn::N=1}
\mathcal{W}(1;x_1,x_2)=\int_{\{x_2<w_1\}}(w_1-x_1)^{-\frac{4}{\kappa}}(w_1-x_2)^{\frac{2-\kappa}{\kappa}}\ud w_1=B\left(\frac{2}{\kappa},\frac{2}{\kappa}\right)\times(x_2-x_1)^{-\frac{2}{\kappa}}.
\end{equation}
Define
\[R_t:=\frac{(g^{\ell}_t(w_1)-W^{\ell}_t)^{-\frac{4}{\kappa}}(g^{\ell}_t(x_2)-W^{\ell}_t)^{\frac{\kappa-2}{2\kappa}}(g^{\ell}_t(w_1)-g^{\ell}_t(x_2))^{\frac{2-\kappa}{\kappa}}\left(g^{\ell}_t\right)'(w_1)}{(g^{\ell}_t(x_2)-W^{\ell}_t)^{\frac{\kappa-6}{2}}}.\]

 Note that there exists a constant $C=C(\eps,w_1)>0$ such that
\[0\le R_{t\wedge\tau_\eps}\le C.\]
Combining with~\cite[Theorem 6]{SchrammWilsonSLECoordinatechanges}, we have that $\{R_{t\wedge\tau_\eps}\}_{t\ge 0}$ is a martingale for $\eta_1$. Moreover, the law of $\ell[0,\tau_\eps]$ weighted by $\{R_{t\wedge\tau_\eps}\}_{t\ge 0}$ equals the conditional law of $\eta_1[0,\tau_\eps]$ given $w_1$. 
Thus, for every bounded and continuous function $g$ on the curve space, we have
\begin{align*}
\E[g(\eta_1[0,\tau_\eps])]=\E[\E[g(\gamma_1[0,\tau_\eps])]\cond w_1]=\int_{\{y<w_1\}}\rho(1;x_1,x_2;w_1)\E\left[g(\ell[0,\tau_\eps])\frac{R_{\tau_\eps}}{R_0}\right]=\E[g(\ell[0,\tau_\eps])].
\end{align*}
By letting $\eps\to 0$, we complete the proof of Lemma~\ref{lem::cascade_gamma} when $N=1$.

The proof when $N\ge 2$ is similar to the proof of Lemma~\ref{lem::cascade_eta}.
Define
\begin{align*}
V_t:=&\prod_{\substack{1\le i\le N-1\\1\le j\le N-M}}(g^\ell_t(w_j)-g^{\ell}_t(x_i))^{-\frac{4}{\kappa}}\prod_{1\le j\le N-M}(g_t^\ell(w_j)-W^\ell_t)^{-\frac{4}{\kappa}}\prod_{1\le j\le N-M}(g^\ell_t(w_j)-g^\ell_t(x_{N+1}))^{\frac{2-\kappa}{\kappa}}\\
&\times\prod_{1\le i\le N-1}(g_t^\ell(x_{N+1})-g_t^\ell(x_i))^{\frac{2}{\kappa}}\times(g_t^\ell(x_{N+1})-W_t^\ell)^{\frac{2}{\kappa}}\prod_{1\le i<j\le N-M}(g_t^\ell(w_j)-g_t^\ell(w_i))^{\frac{8}{\kappa}}\prod_{1\le j\le N-M}\left(g_t^\ell\right)'(w_j).
\end{align*}
Note that there exists $C=C(\eps,w_1,\ldots,w_{N-M})$ such that
\[0\le V_{t\wedge\tau_\eps}\le C.\]
Combining with~\cite[Theorem 6]{SchrammWilsonSLECoordinatechanges}, we have that $\{V_{t\wedge\tau_\eps}\}_{t\ge 0}$ is a martingale for $\ell$. 
Similarly to~\eqref{eqn::truncated_mart_eta}, for every bounded and continuous function $g$ on the curve space, and bounded and continuous function $f$, we have
\begin{align}
&\E[g(\eta_N[0,\tau_\eps])f(g^{\eta_N}_{\tau_\eps}(w_2),\ldots, g^{\eta_N}_{\tau_\eps}(w_{N-M}))]\notag\\
=&\frac{\prod_{1\le i\le N}(x_{N+1}-x_i)^{-\frac{2}{\kappa}}}{\mathcal{W}_N(\bs x; x_{N+1})}\E\left[g(\ell[0,\tau_\eps])V(f,N;g^\ell_{\tau_\eps}(x_1),\ldots,g^\ell_{\tau_\eps}(x_{N-1}),W^\ell_{\tau_\eps},g^\ell_{\tau_\eps}(x_{N+1}))\right],\label{eqn::truncated_mart_gamma}
\end{align}
where we define 
\begin{align}\label{eqn::mart_gamma}
&V(f,N;g^\ell_{t}(x_1),\ldots,g^\ell_{t}(x_{N-1}),W^\ell_{t},g^\ell_{t}(x_{N+1}))\notag\\
:=&\prod_{1\le i\le N-1}(g_t^\ell(x_{N+1})-g_t^\ell(x_i))^{\frac{2}{\kappa}}\times(g_t^\ell(x_{N+1})-W_t^\ell)^{\frac{2}{\kappa}}\notag\\
&\times\int_{\{g^\ell_{t}(x_{N+1})<u_1\cdots<u_{N-M}\}}\prod_{1\le i<j\le N-M}(u_j-u_i)^{\frac{8}{\kappa}}\prod_{\substack{1\le i\le N-1\\1\le j\le N-M}}(u_j-g^{\ell}_t(x_i))^{-\frac{4}{\kappa}}\notag\\
&\times f(u_2,\ldots,u_{N-M})\times \prod_{1\le j\le N-M}(u_j-W^\ell_t)^{-\frac{4}{\kappa}}(u_j-g^\ell_t(x_{N+1}))^{\frac{2-\kappa}{\kappa}}\ud u_j.
\end{align}
When $f=1$, by definition, we have that 
\[V(f,N;g^\ell_{t}(x_1),\ldots, g^\ell_{t}(x_{N-1}), W^\ell_{t},g^\ell_{t}(x_{N+1}))=S_t.\]
This gives the first item.

Next, we will prove that $\{V(f,N;g^\ell_{t}(x_1),\ldots, g^\ell_{t}(x_{N-1}), W^\ell_{t},g^\ell_{t}(x_{N+1}))\}_{0\le t\le \tau}$ is uniformly bounded and derive the terminal value.
For $1\le i\le N-1$ and $2\le j\le N-M$, we have
\[u_j-u_1\le u_j-g_t^\ell(x_{N+1}),\quad u_j-W_t^\ell\ge u_j-g_t^\ell(x_{N+1}),\quad  u_1-g_t^\ell(x_i)\ge g_t^\ell(x_{N+1})-g_t^\ell(x_{i})\]
and
\[u_j-g_t^\ell(x_i)\ge u_j-g_t^\ell(x_{N-1}).\]
Then, we have
\begin{align}\label{eqn::bound_intd_V}
&\prod_{2\le j\le N-M}(u_j-u_1)^{\frac{8}{\kappa}}\times \prod_{\substack{1\le i\le N-1\\1\le j\le N-M}}(u_j-g^{\ell}_t(x_i))^{-\frac{4}{\kappa}}\times\prod_{2\le j\le N-M}(u_j-W^\ell_t)^{-\frac{4}{\kappa}}(u_j-g^\ell_t(x_{N+1}))^{\frac{2-\kappa}{\kappa}}\notag\\
\le&\prod_{2\le j\le N-M}(u_j-g^{\ell}_t(x_{N-1}))^{-\frac{4}{\kappa}(N-1)}(u_j-g^\ell_t(x_{N+1}))^{\frac{6-\kappa}{\kappa}}\times\prod_{1\le i\le N-1}(g_t^\ell(x_{N+1})-g_t^\ell(x_{i}))^{-\frac{4}{\kappa}}.
\end{align}
Thus, similarly to~\eqref{eqn::bound_U_1}, by~\eqref{eqn::bound_intd_V}, we have
\begin{align}
&\int_{\{g^\ell_{t}(x_{N+1})<u_1\cdots<u_{N-M}\}}\prod_{1\le i<j\le N-M}(u_j-u_i)^{\frac{8}{\kappa}}\prod_{\substack{1\le i\le N-1\\1\le j\le N-M}}(u_j-g^{\ell}_t(x_i))^{-\frac{4}{\kappa}}\notag\\
&\times\prod_{1\le j\le N-M}(u_j-W^\ell_t)^{-\frac{4}{\kappa}}(u_j-g^\ell_t(x_{N+1}))^{\frac{2-\kappa}{\kappa}}\ud u_j\notag\\
\le&\prod_{1\le i\le N-1}(g_t^\ell(x_{N+1})-g_t^\ell(x_i))^{-\frac{4}{\kappa}}\notag\\
&\times\int_{\{g^\ell_{t}(x_{N+1})<u_2\cdots<u_{N-M}\}}\prod_{2\le i<j\le N-M}(u_j-u_i)^{\frac{8}{\kappa}}\prod_{2\le j\le N-M}(u_j-g^{\ell}_t(x_{N-1}))^{-\frac{4}{\kappa}(N-1)}(u_j-g^\ell_t(x_{N+1}))^{\frac{6-\kappa}{\kappa}}\ud u_j\notag\\
&\times\int_{\{g^\ell_{t}(x_{N+1})<u_1<u_2\}}(u_1-W^{\ell}_t)^{-\frac{4}{\kappa}}(u_1-g^\ell_t(x_{N+1}))^{\frac{2-\kappa}{\kappa}}\ud u_1\notag\\
\le&\prod_{1\le i\le N-1}(g_t^\ell(x_{N+1})-g_t^\ell(x_i))^{-\frac{4}{\kappa}}\notag\\
&\times\int_{\{g^\ell_{t}(x_{N+1})<u_2\cdots<u_{N-M}\}}\prod_{2\le i<j\le N-M}(u_j-u_i)^{\frac{8}{\kappa}}\prod_{2\le j\le N-M}(u_j-g^{\ell}_t(x_{N-1}))^{-\frac{4}{\kappa}(N-1)}(u_j-g^\ell_t(x_{N+1}))^{\frac{6-\kappa}{\kappa}}\ud u_j\notag\\
&\times\int_{\{g^\ell_{t}(x_{N+1})<u_1\}}(u_1-W^{\ell}_t)^{-\frac{4}{\kappa}}(u_1-g^\ell_t(x_{N+1}))^{\frac{2-\kappa}{\kappa}}\ud u_1\notag\\
=&\prod_{1\le i\le N-1}(g_t^\ell(x_{N+1})-g_t^\ell(x_i))^{-\frac{4}{\kappa}}\times (g_t^\ell(x_{N+1})-g_t^\ell(x_{N-1}))^{\frac{4}{\kappa}(N-M-1)\left(\frac{1}{2}-M\right)}\notag\\
&\times\int_{\{0<u_2\cdots<u_{N-M}<1\}}\prod_{2\le i<j\le n-M}(u_j-u_i)^{\frac{8}{\kappa}}\prod_{2\le j\le N-M}u_j^{\frac{8M-4N+6-\kappa}{\kappa}}(1-u_j)^{\frac{6-\kappa}{\kappa}}\ud u_j\notag\\
&\times B\left(\frac{2}{\kappa},\frac{2}{\kappa}\right)\times(g_t^\ell(x_{N+1})-W_t^\ell)^{-\frac{2}{\kappa}}, \label{eqn::bound_V_1}
\end{align}
where $B(\cdot,\cdot)$ is the Beta function.
Similarly to~\eqref{eqn::bound_U_2}, there exists a constant $C_4>0$ which is independent of $\eps$, such that
\begin{equation}\label{eqn::bound_V_2}
V(f,N;g^\ell_{\tau_\eps}(x_1),\ldots,g^\ell_{\tau_\eps}(x_{N-1}),W^\ell_{\tau_\eps},g^\ell_{\tau_\eps}(x_{N+1}))\le C_4\max f.
\end{equation}
Similarly to~\eqref{eqn::bound_V_1}, by~\eqref{eqn::bound_intd_V} and dominated convergence theorem, we have
\begin{align}\label{eqn::limit_V}
&\lim_{\eps\to 0}V(f,N;g^\ell_{\tau_\eps}(x_1),\ldots,g^\ell_{\tau_\eps}(x_{N-1}),W^\ell_{\tau_\eps},g^\ell_{\tau_\eps}(x_{N+1}))\notag\\
=&B\left(\frac{2}{\kappa},\frac{2}{\kappa}\right)\times\prod_{1\le i\le N-1}(g_{\tau}^\ell(x_{N+1})-g_\tau^\ell(x_i))^{-\frac{2}{\kappa}}\mathcal{Z}_{N-1}(g_\tau^\ell(x_1),\ldots,g_\tau^\ell(x_{N-1}); W^\ell_{\tau})\notag\\
&\times\int_{\{W^\ell_{\tau}<u_2<\cdots<u_{N-M}\}}f(u_2,\ldots,u_{N-M})r(N-1;g^\ell_{\tau}(x_1),\ldots,g^\ell_\tau(x_{N-1}),W^\ell_{\tau};u_2,\ldots,u_{N-M})\prod_{2\le j\le N-M}\ud u_j\notag\\
:=&V(f,N;g^\ell_{\tau}(x_1),\ldots,g^\ell_\tau(x_{N-1}),W^\ell_{\tau}).
\end{align}
Combining~\eqref{eqn::truncated_mart_gamma}--\eqref{eqn::limit_V} together,  letting $\eps\to 0$, by dominated convergence theorem, we have
\begin{align}
&\E[g(\eta_N[0,\tau])f(g^{\eta_N}_{\tau}(w_2),\ldots, g^{\eta_N}_{\tau}(w_{N-M}))]
\notag\\
=&\frac{\prod_{1\le i\le N-1}(x_{N+1}-x_i)^{-\frac{2}{\kappa}}}{\mathcal{W}(N;\bold x,x_{N+1})}\E\left[g(\ell[0,\tau])V(f,N;g^\ell_{\tau}(x_1),\ldots,g^\ell_\tau(x_{N-1}),W^\ell_{\tau})\right].\label{eqn::mart_eta}
\end{align}
Taking $f=1$, we obtain~\eqref{eqn::weight_eta}.  This gives the second item.

Note that~\eqref{eqn::mart_eta}  implies that the density of the conditional law of $(g^{\eta_N}_{\tau}(w_2),\ldots, g^{\eta_N}_{\tau}(w_{N-M}))$ given $\eta_N$ is given by 
$r(N-1;g^{\eta_N}_{\tau}(x_1),\ldots,g^{\eta_N}_\tau(x_{N-1}),W^{\eta_1}_{\tau};\cdot,\ldots,\cdot)$. Combining with the coupling of flow lines, we have that the conditional law of $(g^{\eta_N}_{\tau}(\eta_1),\ldots, g^{\eta_N}_{\tau}(\eta_{N-1}))$ equals the law of random curves under $\PP_{N-1}(g^{\eta_N}_{\tau}(x_1),\ldots,g^{\eta_N}_\tau(x_{N-1}); W^{\eta_N}_{\tau})$. This gives the third item and completes the proof.
\end{proof}
In Lemma~\ref{lem::curve_properties}, we will  use cascade relations in Lemma~\ref{lem::cascade_eta} and Lemma~\ref{lem::cascade_gamma} to prove that $(\eta_1,\ldots,\eta_N)\sim\PP_N$ and $(\gamma_1,\ldots,\gamma_N)\sim\QQ_N$ have the same law. 
\begin{lemma}\label{lem::curve_properties}
\begin{itemize}
\item
We have the following equation on partition functions
\begin{align}\label{eqn::partition_func}
&\left(B\left(\frac{2}{\kappa},\frac{2}{\kappa}\right)\right)^{\one_{\{N\text{ is odd}\}}}\times\mathcal{Z}_N(\bold x; x_{N+1})
=\prod_{1\le i\le N}(x_{N+1}-x_i)^{\frac{2}{\kappa}}\times\mathcal{W}_N(\bold x; x_{N+1}),
\end{align}
where we use the convention that $\mathcal{Z}_1(\cdot; \cdot)=1$.
\item
The law of $(\eta_1,\ldots,\eta_N)\sim\PP_N$ equals the law of $(\gamma_1,\ldots,\gamma_N)\sim\QQ_N$.
\end{itemize}
\end{lemma}
\begin{proof}[Proof of Lemma~\ref{lem::curve_properties}]
We will prove by induction. The case $N=1$ has been proved in Lemma~\ref{lem::cascade_gamma}.

Now, we suppose that Lemma~\ref{lem::curve_properties} holds for $N\le n-1$. Then, we have
\begin{itemize}
\item
By~\eqref{eqn::weight_ell} and~\eqref{eqn::weight_eta} and the induction hypothesis,  we can take the expectation and then obtain~\eqref{eqn::partition_func} when $N=n$.
\item
By~\eqref{eqn::weight_ell},~\eqref{eqn::weight_eta} and the induction hypothesis, when $N=n$,  we have that the law of $\eta_n$ equals the law of $\gamma_n$. Moreover,  by Lemma~\ref{lem::cascade_eta} and Lemma~\ref{lem::cascade_gamma} and the induction hypothesis, the conditional law of $(\ell_1,\ldots,\ell_{n-1})$ given $\ell_n$ equals the conditional law of $(\eta_1,\ldots,\eta_{n-1})$ given $\eta_n$. Thus, we have $\PP_n=\QQ_n$.
\end{itemize}
This completes the induction and hence completes the proof.
\end{proof}
With Lemma~\ref{lem::curve_properties} at hands, we can complete the proof of Theorem~\ref{thm::multiple_odd}, Theorem~\ref{thm::multiple_even}.
\begin{proof}[Proof of Theorem~\ref{thm::multiple_odd} and Theorem~\ref{thm::multiple_even}]
Suppose that $(\gamma_1,\ldots,\gamma_N)$ has the same law as random curves under $\PP_N=\QQ_N$. Recall the properties given in Section~\ref{sec::flow_line}.
\begin{itemize}
\item
If $j$ is even,  the law of $(\gamma_1,\ldots,\gamma_N)$ equals the law of $(\ell_1,\ldots,\ell_N)$ under $\PP_N$. Given $\cup_{i\neq N-j}\ell_{i}$, the conditional law of $\ell_{N-j}$ equals the $\SLE_\kappa\left(\frac{\kappa-6}{2},\frac{\kappa-6}{2}\right)$ in $\Omega_{N-j}$ with force points $\ell_{j-1}(\tau_{j-1})$ and $\ell_{j+1}(\tau_{j+1})$. 
\item
If $j$ is odd, the law of $(\gamma_1,\ldots,\gamma_N)$ equals the law of $(\eta_1,\ldots,\eta_N)$ under $\QQ_N$. Given $\cup_{i\neq N-j}\eta_{i}$, the conditional law of $\eta_{N-j}$ equals the $\SLE_\kappa\left(\frac{\kappa-6}{2},\frac{\kappa-6}{2}\right)$ in $\Omega_{N-j}$ with force points $\eta_{j-1}(\tau_{j-1})$ and $\eta_{j+1}(\tau_{j+1})$. 
\end{itemize}
This proves that the law of $(\gamma_1,\ldots,\gamma_N)$ is a multiple $\SLE_\kappa\left(\frac{\kappa-6}{2},\frac{\kappa-6}{2}\right)$ and completes the proof.
\end{proof}

Next, we come to the proof of Theorem~\ref{thm::uniqueness_multiple}.  The strategy is standard (see~\cite{MillerSheffieldIG2}): constructing Markov chain by using the required conditional law, and then proving that it has the unique stationary measure. See also~\cite[Theorem 1.2]{BeffaraPeltolaWuUniqueness} and~\cite[Theorem 4.2]{ZhanExistenceUniquenessMultipleSLE}.  In~\cite[Theorem 1.2]{BeffaraPeltolaWuUniqueness}, the authors proved a stronger result: they derived the exponential decay of mixing time. In our setting, a crucial lemma~\cite[Lemma 3.6]{BeffaraPeltolaWuUniqueness} does not hold. Thus, it is not obvious that how to obtain such an estimate. In this paper, we will use the same argument as in the proof of~\cite[Theorem 4.2]{ZhanExistenceUniquenessMultipleSLE}. Below, we will explain how to modify it to fit our setting. 

Without loss of generality, we may assume that $\Omega=\HH$ and $x_{N+2}=\infty$. For $\hat{\bold \gamma}=(\hat\gamma_1,\ldots,\hat\gamma_N)\in X_N(\HH;x_1,\ldots,x_{N};x_{N+1},\infty)$, suppose $\Phi:=(\Phi_j=(\Phi_j(1),\ldots,\Phi_j(N))_{j\ge 0}\sim\PP_{\hat{\bold \gamma}}$ is the following Markov chain.
\begin{itemize}
\item
The initial state $\Phi_0$ equals $\hat\gamma$.
\item
The conditional law of $\Phi_n$ given $(\Phi_j)_{0\le j\le n-1}$ can be described as follows. First, sample a uniform distribution $U$ on $\{1,\ldots, N\}$. Denote by $\Omega_U$ the connected component of $\HH\setminus\cup_{i\neq U}\Phi_{n-1}(i)$ whose boundary contains $x_U$. As in the definition of multiple $N$-$\SLE_\kappa\left(\frac{\kappa-6}{2},\frac{\kappa-6}{2}\right)$, denote by $\tau_j$ the hitting time of $\Phi_{n-1}(j)$ at $(x_{N+1},+\infty)$ for $1\le j\le N$. Then, denote by $\eta_U$ the $\SLE_\kappa\left(\frac{\kappa-6}{2},\frac{\kappa-6}{2}\right)$ starting from $x_U$ with two force points $\Phi_{n-1}(U-1)(\tau_{U-1})$ and $\Phi_{n-1}(U+1)(\tau_{U+1})$. Set $\Phi_n=(\Phi_{n-1}(1),\ldots,\Phi_{n-1}(U-1),\eta_U,\Phi_{n-1}(U+1),\ldots,\Phi_{n-1}(N))$.
\end{itemize}
We call $T_j$  a tube contain $x_j$, if   $T_j$ is the bounded domain surrounded by $\alpha_j[0,1]\cup(\beta_j(1),\alpha_j(1))\cup \beta_j[0,1]\cup (\alpha_j(0),\beta_j(0))$, where
$\alpha_j$ and $\beta_j$ are two simple curves on $\HH$, and $\alpha_j$ starts from $(x_{j-1},x_j)$ and ends at $(x_{N+1},+\infty)$ and $\beta_j$ starts from $(x_j,x_{j+1})$ and ends at $(x_{N+1},+\infty)$ (here we denote by $x_0:=-\infty$).
Denote by $\mathcal T_j$ the set of simple curves contained in $T_j$, which start from $x_j$ and end at $(\beta_j(1),\alpha_j(1))$. Choose $N$ disjoint tubes $T_1,\ldots, T_N$ such that $T_j$ contains $x_j$ for $1\le j\le N$.
For $1\le j\le N$, denote by $\mu_j$ the law of $\SLE_\kappa\left(\frac{\kappa-6}{2},\frac{\kappa-6}{2}\right)$ curve on $T_j$ which starts from $x_j$ with force points $x_1^R=\beta_j(1)$ on its right and $x_1^L=\alpha_j(1)$ on its left.  
Define a measure $\phi$ on $(\gamma_1,\ldots,\gamma_N)\in X_N(\HH;x_1,\ldots,x_{N};x_{N+1},\infty)$ by 
\[\phi[\cdot]:=\bigotimes_{1\le j\le N}\mu_j\left[\cdot\right].\]
It suffices to prove that for every measurable set $A\subset X_N(\HH;x_1,\ldots,x_{N};x_{N+1},\infty)$ with $\phi[A]>0$,  and for every $\hat\gamma\in X_N(\HH;x_1,\ldots,x_{N};x_{N+1},\infty)$, we have that 
\begin{equation}\label{eqn::transition}
\PP_{\hat\gamma}[\Phi_{3N}\in A]>0.
\end{equation}
The proof of~\eqref{eqn::transition} depends on the following lemma.
\begin{lemma}\label{lem::diff_domain}
Fix $1\le j\le N$. Suppose $\hat T_j$ is a simply connected domain which contains $T_j$ and choose $y_1,y_2\in\partial\hat T_j\cap (x_{N+1},+\infty)$ such that $y_1<\alpha_j(1)<\beta_j(1)<y_2$. Denote by $\ell$ (resp. $\hat\ell$) the $\SLE_\kappa\left(\frac{\kappa-6}{2},\frac{\kappa-6}{2}\right)$ on $T_j$ (resp. $\hat T_j$) with two forces points $x_1^R=\beta_j(1)$ (resp. $x_1^R=y_1$) on its right and $x_1^L=\alpha_j(1)$ (resp. $x_1^L=y_2$) on its left. Then, the law of $\ell$ is absolutely continuous with respect to $\hat\ell$ when restricted on $T_j$. In particular, we have 
\begin{equation}\label{eqn::posi_ex}
\PP[\hat\ell\subset T_j]>0.
\end{equation}
\end{lemma}
\begin{proof}
This is same as the proof of~\cite[Lemma 3.1]{Yureversibility}: Consider the flow line coupling $(\ell,\Gamma)$ (resp. $(\hat\ell,\hat\Gamma)$) on $T_j$ (resp. $\hat T_j$). The law of $\Gamma$ is absolutely continuous with respect to $\hat\Gamma$ restricted on $T_j$. Since flow line is a measurable function of the underlying field, we obtain the absolute continuity and~\eqref{eqn::posi_ex}.
\end{proof}
Now, we come back to the proof of Theorem~\ref{thm::uniqueness_multiple} and Corollary~\ref{coro::hitting_law}.
\begin{proof}[Proof of Theorem~\ref{thm::uniqueness_multiple}]
We only need to prove~\eqref{eqn::transition}. The proof is almost the same as the proof of~\cite[Theorem 4.2]{ZhanExistenceUniquenessMultipleSLE} and we only give the outline below. It suffices to construct a sequence of measurable sets $A_0,\ldots A_{3N}$ such that $A_0=\prod_{1\le j\le N}\{\gamma_j\}$ and $A_{3N}=A$, and for $1\le j\le 3N$, for every $\hat\chi\in A_{j-1}$, we have that
\begin{equation}\label{eqn::transition_prob}
\PP_{\hat\chi}[\Phi_1\in A_j]>0.
\end{equation}
The construction consists of three steps.
\begin{itemize}
\item
In the first step, we can choose another $N$ disjoint tubes $S_1,\ldots,S_N$ such that $S_j$ contains $x_j$ for $1\le j\le N$ and $S_j$ is disjoint from $\gamma_k$ and $T_k$ for $1\le j<k\le N$. Define $A_j:=\prod_{1\le i\le j}\mathcal S_j\times\prod_{j+1\le i\le N}\{\gamma_i\}$. By Lemma~\ref{lem::diff_domain}, we obtain~\eqref{eqn::transition_prob} for $1\le j\le N$.
\item
In the second step, let $A_{N+j}=\prod_{1\le i\le N-j}\mathcal S_i\times\prod_{N-j+1\le i\le N}\mathcal T_i$. Still by Lemma~\ref{lem::diff_domain}, we obtain~\eqref{eqn::transition_prob} for $N+1\le j\le 2N$.
\item
In the third step, inductively, we define $A_{2N+j-1}$ (for $2\le j\le N$) to be the set of $(\chi(1),\ldots,\chi(N))\in\prod_{1\le i\le N}\mathcal T_i$ such that for $\ell\sim\mu_j$, we have
\begin{equation}\label{eqn::cons_A}
\mu_j[(\chi(1),\ldots,\chi({j-1}),\ell,\chi(r_{j+1}),\ldots, \chi(N))\in A_{N+2j}]>0.
\end{equation}
Then, for $1\le j\le N$, we have that $A_{2N+j}$ has the form $\mathcal C_j\times\prod_{j+1\le i\le N}\mathcal T_i$ for some measurable set $\mathcal C_j\subset\prod_{1\le i\le j}\mathcal T_i$. For every $\hat\chi\in A_{2N+j-1}$, let $\hat T_j$ be the connected component of $\HH\setminus\cup_{i\neq j}\chi(i)$ which contains $x_j$ and choose $y_1=\chi(j-1)(\tau_{j-1})$ and $y_2=\chi(j+1)(\tau_{j+1})$ in Lemma~\ref{lem::diff_domain}. Recall that we denote by $\hat\ell$ the $\SLE_\kappa\left(\frac{\kappa-6}{2},\frac{\kappa-6}{2}\right)$ on $\hat T_j$ with two forces points  $x_1^R=y_1$ on its right and  $x_1^L=y_2$ on its left. Then, by Lemma~\ref{lem::diff_domain} and~\eqref{eqn::cons_A}, we have
\[\PP_{\hat\chi}[\Phi_1\in A_{2n+j}]=\frac{1}{N}\PP[(\chi(1),\ldots,\chi({j-1}),\hat\ell,\chi(r_{j+1}),\ldots, \chi(N))\in A_{N+2j}]>0.\]
Thus,  we obtain~\eqref{eqn::transition_prob} for $2N+1\le j\le 3N$.
\end{itemize}
This completes the proof.
\end{proof}
\begin{proof}[Proof of Corollary~\ref{coro::hitting_law}]
By Theorem~\ref{thm::uniqueness_multiple}, the law of $(\gamma_1,\ldots,\gamma_N)$ equals the law of random curves in Theorem~\ref{thm::multiple_odd} and Theorem~\ref{thm::multiple_even}. Combining with the properties given in Section~\ref{sec::flow_line}, we complete the proof.
\end{proof}

\section{Proof of Theorem~\ref{thm::same_point_odd}--Corollary~\ref{coro::same_point_hitting_law}}
\label{sec::jacobi_SLE}
In this section, we will prove Theorem~\ref{thm::same_point_odd}--Corollary~\ref{coro::same_point_hitting_law}.
\begin{proof}[Proof of Theorem~\ref{thm::same_point_odd} and Theorem~\ref{thm::same_point_even}]
The proof is almost the same as the proof of Theorem~\ref{thm::multiple_odd} and Theorem~\ref{thm::multiple_even}: Note that in Lemma~\ref{lem::cascade_eta}--Lemma~\ref{lem::curve_properties}, we do not essentially use the assumption that $x_1<\ldots<x_N$. The same proof still works when we let $x_1=\cdots=x_N=x$ and $x_{N+1}=u$. This completes the proof.
\end{proof}
Note that the proof of Theorem~\ref{thm::uniqueness_multiple} can not be applied when the random curves start from the same point: we can not choose disjoint tubes. Based on  Theorem~\ref{thm::uniqueness_multiple}, we will use an approximation procedure.
\begin{proof}[Proof of Theorem~\ref{thm::same_point_uniqueness}]
We will prove by induction. Suppose Theorem~\ref{thm::same_point_uniqueness} holds when $N\le n-1$.  Next, we prove Theorem~\ref{thm::same_point_uniqueness} when $N=n$. Suppose $(\gamma_1,\ldots, \gamma_n)$ is a multiple $n$-$\SLE\left(\frac{\kappa-6}{2},\frac{\kappa-6}{2}\right)$ on $X_n(\HH;x;u,\infty)$. By induction hypothesis, it suffices to determine the law of $\gamma_n$. We parameterise curves by its half-plane capacity.
For $0<\eps<\eps_0<\frac{1}{2}(u-x)$, denote by $\sigma_j^\eps$ the hitting time of $\gamma_j$ at $\partial B(x,\eps)$ for $1\le j\le n-1$ and by $\sigma_n^{\eps_0}$ the hitting time of $\gamma_n$ at $\partial B(x,\eps_0)$.
 We generate $\gamma_1[0,\sigma_1^\eps],\ldots,\gamma_{n-1}[0,\sigma_{n-1}^\eps]$ and $\gamma_n[0,\sigma_{\eps_0}]$. Consider the following random curves.
 \begin{itemize}
\item
Denote by $g^{\eps, \eps_0}$ the conformal map from $\HH\setminus\left(\cup_{1\le j\le n-1}\gamma_j[0,\sigma_j^\eps]\cup\gamma_n[0,\sigma_n^{\eps_0}]\right)$ onto $\HH$ normalised at infty. 
Define $\gamma_n^{\eps,\eps_0}(t):=g^{\eps, \eps_0}(\gamma_n(\sigma_n^{\eps_0}+t))$ for $t\ge 0$. Denote by $x_j=x_j^{\eps,\eps_0}:=g^{\eps, \eps_0}(\gamma_j(\sigma_j^\eps))$ for $1\le j\le n-1$ and by $x_n=x_n^{\eps,\eps_0}:=g^{\eps, \eps_0}(\gamma_n(\sigma_n^{\eps_0}))$ and by $x_{n+1}=x_{n+1}^{\eps,\eps_0}:=g^{\eps, \eps_0}(u)$. 
Denote by $\ell^{\eps,\eps_0}$ the $\SLE_\kappa\left(2,\ldots,2;\frac{\kappa-6}{2}\right)$ curve starting from $x_n$ with force points $x_1^L=x_{n-1},\ldots, x_{n-1}^L=x_1$ on its left and $x_1^R=x_{n+1}$ on its right.
Denote by $(W^{\eps,\eps_0}_t:t\ge 0)$ the driving function of $\ell^{\eps,\eps_0}$ and by $(g_t^{\eps,\eps_0}:t\ge 0)$ the corresponding conformal maps.
\item
Denote by $g^{\eps_0}$ the conformal map from $\HH\setminus\gamma_n[0,\sigma_n^{\eps_0}]$ onto $\HH$ normalised at infty. 
Define $\gamma_n^{\eps_0}(t):=g^{\eps_0}(\gamma_n(\sigma_n^{\eps_0}+t))$ for $t\ge 0$. 
Denote by $x^{-,\eps_0}:=g^{\eps_0}(x^-)$ and by $x^{\eps_0}:=g^{\eps_0}(\gamma_n(\sigma_n^{\eps_0}))$ and by $u^{\eps_0}:=g^{\eps_0}(u)$. 
Denote by $\ell^{\eps_0}$ the $\SLE_\kappa\left(2(n-1);\frac{\kappa-6}{2}\right)$ curve starting from $x^{\eps_0}$ with force points $x_1^L=x^{-,\eps_0}$ on its left and $x_1^R=u^{\eps_0}$ on its right.
Denote by $(W^{\eps_0}_t:t\ge 0)$ the driving function of $\ell^{\eps_0}$ and by $(g_t^{\eps_0}:t\ge 0)$ the corresponding conformal maps.
\end{itemize}
Denote by $\PP^{\eps,\eps_0}$ the conditional law of $\gamma_n^{\eps,\eps_0}$ given $\gamma_1[0,\sigma_1^\eps],\ldots,\gamma_{n-1}[0,\sigma_{n-1}^\eps]$ and $\gamma_n[0,\sigma_{\eps_0}]$.
By Theorem~\ref{thm::uniqueness_multiple}, under  $\PP^{\eps,\eps_0}$, the law of $\gamma_n^{\eps,\eps_0}$ equals the law of $\ell_n$ under $\PP_n(x_1,\ldots, x_n; x_{n+1})$. For $\eps'>0$, denote by $\tau^{\eps,\eps_0}_{\eps'}$ the hitting time of curves at the union of $\partial B(x_n,\frac{1}{\eps'})$ and the $\eps'$-neighbourhood of $(x_{n+1},+\infty)$. Using the same notations as in Lemma~\ref{lem::cascade_eta}, denote by 
\[U^{\eps,\eps_0}_{t}:=\mathcal Z_n(g^{\eps,\eps_0}_{t}(x_1),\ldots, g^{\eps,\eps_0}_{t}(x_{n-1}), W^{\eps,\eps_0}_{t};g^{\eps,\eps_0}_{t}(x_{n+1})).\]
Similarly to~\eqref{eqn::bound_U_1}, there exists a constant $C_1=C_1(x,u)>0$, such that for $0\le t\le \tau^{\eps,\eps_0}_{\eps'}$, we have
\begin{equation}\label{eqn::uni_boundaux1}
U^{\eps,\eps_0}_{t}\le C_1.
\end{equation}
For $1\le i\le n-1$ and $1\le j\le M$, we have
\[u_j-x_n\le u_j-x_1\quad\text{and}\quad u_j-x_i\le u_j-x_1.\]
By the same computation in~\eqref{eqn::bound_U_1}, we have
\[U^{\eps,\eps_0}_{0}\ge (x_{n+1}-x_1)^{\frac{4}{\kappa}M(M-n+\frac{1}{2})}\int_{0<u_1<\cdots<u_M<1}\prod_{1\le i<j\le M}(u_j-u_i)^{\frac{8}{\kappa}}\prod_{1\le j\le M}u_j^{\frac{4n-8M+2-\kappa}{\kappa}}(1-u_j)^{\frac{6-\kappa}{\kappa}}\ud u_j\]
Thus, by the choice of $\eps$ and $\eps_0$, there exists a constant $C_2=C_2(x,u)>0$, such that
\begin{equation}\label{eqn::uni_boundaux2}
U^{\eps,\eps_0}_{0}\ge C_2.
\end{equation}
Combining~\eqref{eqn::uni_boundaux1} and~\eqref{eqn::uni_boundaux2} and~\eqref{eqn::weight_elleps} together, if define $M_t:=\frac{U^{\eps,\eps_0}_{t}}{U^{\eps,\eps_0}_0}$, we have that $\{M_{t\wedge\tau^{\eps,\eps_0}_{\eps'}}\}_{t\ge 0}$ is a bounded martingale for $\ell^{\eps,\eps_0}$. Moreover, for every continuous function $g$, we have
\begin{equation}\label{eqn::uni_aux2}
\E^{\eps,\eps_0}[g\left(\gamma_n^{\eps,\eps_0}[0,t\wedge\tau^{\eps,\eps_0}_{\eps'}]\right)]=\E\left[g\left(\ell^{\eps,\eps_0}[0,t\wedge\tau^{\eps,\eps_0}_{\eps'}]\right)M_{t\wedge\tau^{\eps,\eps_0}_{\eps'}}\right].
\end{equation}
By the convergence of $g^{\eps,\eps_0}$ to $g^{\eps_0}$ when $\eps\to 0$, we have
\begin{equation}\label{eqn::uni_aux1}
\gamma_n^{\eps,\eps_0}\to \gamma_n^{\eps_0} \text{ locally uniformly}, \quad x_j^{\eps,\eps_0}\to x^{-,\eps_0}\text{ for }1\le j\le n-1, \quad x_n^{\eps,\eps_0}\to x^{\eps_0} \quad\text{and}\quad x_{n+1}^{\eps,\eps_0}\to u^{\eps_0}. 
\end{equation}
By the same proof of~\cite[Proposition 4.10]{LWlevelline}, we have
\begin{equation}\label{eqn::uni_aux3}
W^{\eps,\eps_0}\to W^{\eps_0} \text{ locally uniformly}.
\end{equation}
For $v_1<v_2<v_3$, define
\[U(v_1,v_2,v_3):=\int_{\{v_3<u_1\cdots<u_M\}}\prod_{1\le i<j\le M}(u_j-u_i)^{\frac{8}{\kappa}}\prod_{1\le j\le M}(u_j-v_1)^{-\frac{4}{\kappa}(n-1)}(u_j-v_2)^{-\frac{4}{\kappa}}(u_j-v_3)^{\frac{6-\kappa}{\kappa}}\ud u_j.\]
Denote by $\tau^{\eps_0}_{\eps'}$ the hitting time of curves at the union of $\partial B(x^{\eps_0},\frac{1}{\eps'})$ and the $\eps'$-neighbourhood of $(u^{\eps_0},+\infty)$. For $0\le t\le \tau^{\eps_0}_{\eps'}$, define $N_t:=\frac{U(g^{\eps_0}_{t}(x^{-,\eps_0}),W^{\eps_0}_{t}, g^{\eps_0}_{t}(u^{\eps_0}))}{U(x^{-,\eps_0},x^{\eps_0},u^{\eps_0})}$. Since $\ell^{\eps,\eps_0}[0,\tau^{\eps,\eps_0}_{\eps'}]$ converges to $\ell^{\eps_0}[0,\tau^{\eps_0}_{\eps'}]$ in law, combining~\eqref{eqn::uni_boundaux1}--\eqref{eqn::uni_aux3} together and by dominated convergence theorem, by letting $\eps\to 0$, we have
\begin{equation}\label{eqn::uni_aux4}
\E[g\left(\gamma_n^{\eps_0}[0,t\wedge\tau^{\eps_0}_{\eps'}]\right)]=\E\left[g\left(\ell^{\eps_0}[0,t\wedge\tau^{\eps_0}_{\eps'}]\right)N_{t\wedge\tau^{\eps_0}_{\eps'}}\right].
\end{equation}
Since $\eps'$ is arbitrary, by Girsanov's theorem, before hitting $(u^{\eps_0},+\infty)$, the driving function $W^{\eps_0}$ satisfies the following SDE
\[\ud W_t^{\eps_0}=\sqrt\kappa \ud B_t+\left(\frac{2(n-1)}{W_t^{\eps_0}-g_t^{\eps_0}(x^{-,\eps_0})}+\frac{(\kappa-6)/2}{W_t^{\eps_0}-g_t^{\eps_0}(u^{\eps_0})}+\frac{\kappa}{2}\partial_{v_2} \log U(g^{\eps_0}_t(x^{-,\eps_0}),W^{\eps_0}_t, g^{\eps_0}_t(u^{\eps_0}))\right)\ud t\]
where $B$ is the standard Brownian motion. Hence, before hitting $(u^{\eps_0},+\infty)$, we have
\begin{equation}\label{eqn::uni_aux4}
W_t^{\eps_0}=W_0^{\eps_0}-2(n-1)g_t^{\eps_0}(x^{-,\eps_0})-\frac{\kappa-6}{2} g_t^{\eps_0}(u^{\eps_0})+\sqrt\kappa B_t+\frac{\kappa}{2}\int_0^t \partial_{v_2} \log U(g^{\eps_0}_s(x^{-,\eps_0}),W^{\eps_0}_s, g^{\eps_0}_s(u^{\eps_0}))\ud s,
\end{equation}
Denote by $(W_t: t\ge 0)$ the driving function of $\gamma_n$ and by $(g_t:t\ge 0)$ the corresponding conformal maps. Then, we have
\[W_t^{\eps_0}=W_{t+\sigma_n^{\eps_0}},\quad g^{\eps_0}=g_{\sigma_n^{\eps_0}},\quad{and}\quad g^{\eps_0}_t=g_{\sigma_n^{\eps_0}+t}\circ g^{-1}_{\sigma_n^{\eps_0}}.\]
Combining with~\eqref{eqn::uni_aux4}, after $\sigma_n^{\eps_0}$ and before hitting $(u,+\infty)$, we have
\begin{equation}\label{eqn::uni_aux5}
W_t=W_{\sigma_n^{\eps_0}}-2(n-1)g_t(x^-)-\frac{\kappa-6}{2} g_t(u)+\sqrt\kappa \left(B_t-B_{\sigma_n^{\eps_0}}\right)+\frac{\kappa}{2}\int_{\sigma_n^{\eps_0}}^t \partial_{v_2} \log U(g_s(x^-),W_s, g_s(u))\ud s.
\end{equation}
By letting $\eps_0\to 0$, by monotone convergence theorem, almost surely, we have
\[\lim_{\eps_0\to 0}\int_{\sigma_n^{\eps_0}}^t \partial_{v_2} \log U(g_s(x^-),W_s, g_s(u))\ud s=\int_{0}^t \partial_{v_2} \log U(g_s(x^-),W_s, g_s(u))\ud s.\]
Plugging into~\eqref{eqn::uni_aux5}, almost surely, we have
\begin{equation}\label{eqn::SDE_gamma}
W_t=x-2(n-1)g_t(x^-)-\frac{\kappa-6}{2} g_t(u)+\sqrt\kappa B_t+\frac{\kappa}{2}\int_{0}^t \partial_{v_2} \log U(g_s(x^-),W_s, g_s(u))\ud s.
\end{equation}
Denote by $\ell$ the $\SLE_\kappa\left(2(n-1);\frac{\kappa-6}{2}\right)$ starting from $x$ with force points $x_1^L=x^-$ on its left and $x_1^R=u$ on its right. Then,~\eqref{eqn::SDE_gamma} implies that the law of $\gamma_n$ is absolutely continuous with respect to $\ell$ and determines the Radon-Nikodym derivative. Hence, the law of $\gamma_n$ is uniquely determined by~\eqref{eqn::SDE_gamma}.
This completes the proof.
\end{proof} 
\begin{proof}[Proof of Corollary~\ref{coro::same_point_hitting_law}]
The descriptions of hitting points of flow lines given in Section~\ref{sec::flow_line} still holds for flow lines in Theorem~\ref{thm::same_point_odd} and Theorem~\ref{thm::same_point_even}. Combining Theorem~\ref{thm::same_point_odd}--Theorem~\ref{thm::same_point_uniqueness} together, we complete the proof.
\end{proof}


\section{Proof of Theorem~\ref{thm::Ising_limit}}
\label{sec::Ising_conv}
In this section, we will prove Theorem~\ref{thm::Ising_limit}.  Recall that we denote by $\gamma_j^\delta$ the interface starting from $x_j^\delta$ for $1\le j\le N$. We first recall the convergence of the law of $\gamma_j^\delta$ as $\delta\to 0$.
\begin{lemma}\label{lem::conv_aux}
Use the same notations as in Theorem~\ref{thm::Ising_limit}. For $1\le j\le N$, the law of $\gamma_j^\delta$ converges as $\delta\to 0$. Denote by $\gamma_j$ the limit. Denote by $(W_t^{\gamma_j}:t\ge 0)$ the driving function of $\phi(\gamma_j)$ and by $(g_t^{\gamma_j}:t\ge 0)$ the corresponding conformal map. Denote by $T$ the hitting time of $\gamma_j$ at the union of $\{x_1,\ldots,x_{j-1},x_{j+1},\ldots,x_{N},x_{N+1},x_{N+2}\}$ and the free boundary $(x_{N+1}x_{N+2})$. Then, before time $T$, the driving function $W^{\gamma_j}$ satisfies the SDE
\begin{equation}\label{eqn::local_driving}
\ud W_t^{\gamma_j}=\sqrt 3\ud B_t+\partial_j \log R(g^{\gamma_j}_t(\phi(x_1)),\ldots,g^{\gamma_j}_t(\phi(x_{j-1})), W_t^{\gamma_j}, g^{\gamma_j}_t(\phi(x_{j+1})), \ldots,g^{\gamma_j}_t(\phi(x_N));g^{\gamma_j}_t(\phi(x_{N+1})))\ud t,
\end{equation}
where $B$ is the standard Brownian motion.
\end{lemma}
\begin{proof}
The conclusion is the combination of Lemma 3.4 and Lemma 3.6 in~\cite{FengWuYangIsing} after mapping from $\mathbb D$ to $\HH$. Below, we explain it briefly.
\begin{itemize}
\item
The convergence of the driving function has been proved in~\cite[Theorem 3.1]{IzyurovObservableFree}. Note that in~\cite[Theorem 3.1]{IzyurovObservableFree}, we only need the convergence of $\{(\Omega_\delta;x_1^\delta,\ldots,x_N^\delta;x_{N+1}^\delta,x_{N+2}^\delta)\}_{\delta>0}$ to $(\Omega;x_1,\ldots,x_N;x_{N+1},x_{N+2})$ in the Carath\'eodory sense.
\item
The proof of $\gamma_j^\delta$ under the metric~\ref{eqn::curve_metric} is similar to the proof of~\cite[Theorem 5.1 and Proposition 6.1]{KarrilaMultipleSLELocalGlobal}.
\end{itemize}
This completes the proof.
\end{proof}
Note that in Lemma~\ref{lem::conv_aux}, we do not prove that $\gamma_j$ hits $\partial\Omega$ only at its two ends for $1\le j\le N$, which is crucial in the proof of Theorem~\ref{thm::Ising_limit}. In Lemma~\ref{lem::conv_aux1}, we will give a uniform control of such events in discrete setting for $\delta>0$.
\begin{lemma}\label{lem::conv_aux1}
Fix $1\le j\le N$. For every $\eps>0$ and $\delta>0$, denote by $\sigma_{j,\eps}^\delta$ (resp. $\tau_{j,\eps}^\delta$) the hitting time of $\gamma_j^\delta$ at the $\eps$-neighbourhood of $\{x_1^\delta,\ldots, x_{j-1}^\delta,x_{j+1}^\delta,\ldots,x_N^\delta\}$ (resp. $(x_{N+1}^\delta x_{N+2}^\delta)$). Then, there exists constants $C_1>0$ and $C_2>0$, such that
\begin{equation}\label{eqn::error_estimate1}
\PP[\sigma_{j,\eps}^\delta<\infty, \text{ and }\gamma_j^\delta\text{ ends at }(x_{N+1}^\delta x_{N+2}^\delta)]\le C_1\eps^{C_2}.
\end{equation}
and
\begin{equation}\label{eqn::error_estimate2}
\PP[\tau_{j,\eps}^\delta<\infty,\text{ and then }\gamma_j^\delta\text{ hits }\partial B(\gamma_j^\delta(\tau_{j,\eps}^\delta),\sqrt\eps)\text{ before hitting }(x_{N+1}^\delta x_{N+2}^\delta)] \le C_1\eps^{C_2}
\end{equation}
and
\begin{equation}\label{eqn::error_estimate3}
\PP\left[\gamma_j^\delta\text{ hits }\partial B(x_{N+1}^\delta,\eps)\cup \partial B(x_{N+2}^\delta,\eps)\right] \le C_1\eps^{C_2}
\end{equation}
\end{lemma}
\begin{proof}
The proof is based on the standard argument using Lemma~\ref{lem::RSW}. For completeness, we provide a proof of~\eqref{eqn::error_estimate1} when $j=N$. The other cases and~\eqref{eqn::error_estimate2} and~\eqref{eqn::error_estimate3} can be proved similarly.
We may assume that $N$ is even such that the spins on the left of $\gamma_N^\delta$ equal $+1$ and the spins on the right of  $\gamma_N^\delta$ equal $-1$.

On the event $\{\gamma_j^\delta\text{ ends at }(x_{N+1}^\delta x_{N+2}^\delta)\}$, there is a $+1$ cluster connecting $(x_{N-1}^\delta x_N^\delta)$ to $(x_{N+1}^\delta x_{N+2}^\delta)$. Denote by $\LC_N^\delta$ the leftmost one (the precise choice is not important). For $0\le\ell\le -\log_2\sqrt\eps$, we define the following domains.
\begin{itemize}
\item
Denote by $C^\delta_{\ell,k}$ the connected component of $(B(x_j^\delta, 2^\ell\eps)\setminus B(x_j^\delta, 2^{\ell-1}\eps))\setminus \LC_N^\delta$ which contains parts of $\LC_N^\delta$ on its boundary, for $1\le k\le k_\ell$. Denote by $\alpha_{\ell,k}^\delta$ and $\beta_{\ell,k}^\delta$ two crossings of $\LC_N^\delta$ on the boundary of $C^\delta_{\ell,k}$. 
Denote by $\xi_{\ell,k}$ the boundary condition such that it equals $+1$ along $\alpha_{\ell,k}^\delta$ and $\beta_{\ell,k}^\delta$, and it equals $-1$ along $\partial C^\delta_{\ell,k}\setminus (\alpha_{\ell,k}^\delta\cup\beta_{\ell,k}^\delta)$.
\item
Denote by $\alpha_\ell^\delta$ the first crossing of $\LC_N^\delta$ through $B(x_j^\delta, 2^\ell\eps)\setminus B(x_j^\delta, 2^{\ell-1}\eps)$ and by $\beta_\ell^\delta$ the last crossing. Denote by $C_\ell^\delta$ the connected component of $(B(x_j^\delta, 2^\ell\eps)\setminus B(x_j^\delta, 2^{\ell-1}\eps))\setminus \LC_N^\delta$ which contains $\alpha_\ell^\delta$ and $\beta_\ell^\delta$ on its boundary and contains $C^\delta_{\ell,k}$ for $1\le k\le k_\ell$. Denote by $\xi_{\ell}$ the boundary condition such that it equals $+1$ along $\alpha_{\ell}^\delta$ and $\beta_{\ell}^\delta$, and it equals $-1$ along $\partial C^\delta_{\ell}\setminus (\alpha_{\ell}^\delta\cup\beta_{\ell}^\delta)$.
\end{itemize}
Note that
\begin{align}\label{eqn::RSW_aux1}
&\{\sigma_{j,\eps}^\delta<\infty, \text{ and }\gamma_j^\delta\text{ ends at }(x_{N+1}^\delta x_{N+2}^\delta)\}\notag\\
\subset& \cap_{0\le\ell\le -\frac{1}{2}\log_2\sqrt\eps}\left(\cap_{1\le k\le k_\ell}\{\text{there exists }+1 \text{ crossing between }\alpha_{2\ell,k}^\delta\text{ and }\beta_{2\ell,k}^\delta\text{ on }C^\delta_{2\ell,k}\}\right)^c.
\end{align}
By domain Markov property and monotonicity, we have that
\begin{align}\label{eqn::RSW_aux2}
&\PP[\cap_{0\le\ell\le -\frac{1}{2}\log_2\sqrt\eps}\left(\cap_{1\le k\le k_\ell}\{\text{there exists }+1 \text{ crossing between }\alpha_{2\ell,k}^\delta\text{ and }\beta_{2\ell,k}^\delta\text{ on }C^\delta_{2\ell,k}\}\right)^c]\notag\\
=&\PP\left[\prod_{0\le\ell\le -\frac{1}{2}\log_2\sqrt\eps}\left(1-\prod_{1\le k\le k_\ell}\PP^{\sigma |_{\partial C^\delta_{\ell,k}}}_{C^\delta_{\ell,k}}\left[\text{there exists }+1 \text{ crossing between }\alpha_{2\ell,k}^\delta\text{ and }\beta_{2\ell,k}^\delta\text{ on }C^\delta_{2\ell,k}\right]\right)\right]\notag\\
\le &\PP\left[\prod_{0\le\ell\le -\frac{1}{2}\log_2\sqrt\eps}\left(1-\prod_{1\le k\le k_\ell}\PP^{\xi_{\ell,k}}_{C^\delta_{\ell,k}}\left[\text{there exists }+1 \text{ crossing between }\alpha_{2\ell,k}^\delta\text{ and }\beta_{2\ell,k}^\delta\text{ on }C^\delta_{2\ell,k}\right]\right)\right].
\end{align}
Still by domain Markov property and monotonicity, we have that
\begin{align}\label{eqn::RSW_aux3}
&\prod_{1\le k\le k_\ell}\PP^{\xi_{\ell,k}}_{C^\delta_{\ell,k}}\left[\text{there exists }+1 \text{ crossing between }\alpha_{2\ell,k}^\delta\text{ and }\beta_{2\ell,k}^\delta\text{ on }C^\delta_{2\ell,k}\right]\notag\\
\ge& \PP^{\xi_{\ell}}_{C^\delta_{\ell}}\left[\text{there exists }+1 \text{ crossing between }\alpha_{2\ell}^\delta\text{ and }\beta_{2\ell}^\delta\text{ on }C^\delta_{2\ell}\right]
\end{align}
Combining~\eqref{eqn::RSW_aux1}--\eqref{eqn::RSW_aux3} and Lemma~\ref{lem::RSW} together, we obtain~\eqref{eqn::error_estimate1}. This completes the proof.
\end{proof}
By Lemma~\ref{lem::conv_aux}, we can also obtain the tightness of $\{(\gamma_1^\delta,\ldots,\gamma_N^\delta)\}_{\delta>0}$. Suppose $(\gamma_1,\ldots,\gamma_N)$ is a subsequential limit and suppose $(\gamma_1^{\delta_n},\ldots,\gamma_N^{\delta_n})$ converges to $(\gamma_1,\ldots,\gamma_N)$ in law as $n\to\infty$. By Skorokhod’s representation theorem, we can construct a coupling $\mathcal Q$ of $\{(\gamma_1^{\delta_n},\ldots,\gamma_N^{\delta_n})\}_{n\ge 1}$ and $(\gamma_1,\ldots,\gamma_N)$ such that $\gamma_j^{\delta_n}$ converges to $\gamma_j$ almost surely as $n\to\infty$. 
\begin{lemma}\label{lem::conv_aux2}
Denote by $\sigma$ (resp. $\tau$) the hitting time of $\gamma_N$ at $\{x_1,\ldots, x_{N-1}\}$ (resp. $(x_{N+1}x_{N+2})$). Under $\mathcal Q$, we have that $\one_{\{\gamma_N^{\delta_n}\text{ ends at }(x_{N+1}^\delta x_{N+2}^\delta)\}}$ converges to $\one_{\{\tau<\sigma\}}$ in probability.  Moreover, we have that $\gamma_j$ hits $\partial\Omega$ only at its two ends for $1\le j\le N$ and $\gamma_j$ does not hit $\{x_{N+1},x_{N+2}\}$.
\end{lemma}
\begin{proof}
We will use the same notations as in Lemma~\ref{lem::conv_aux1}. Choose $\eps$ small enough. We have the following relations
\begin{align*}
&\{d(\gamma_N^{\delta_n},\gamma_N)\le\eps\}\cap\left(\{\gamma_N^{\delta_n}\text{ ends at }(x_{N+1}^\delta x_{N+2}^{\delta_n})\}\setminus\{\tau<\sigma\}\right)\\
\subset&\{\sigma_{j,\eps}^{\delta_n}<\infty, \text{ and }\gamma_j^{\delta_n}\text{ ends at }(x_{N+1}^{\delta_n}x_{N+2}^{\delta_n})\}\cup\{\gamma_N^{\delta_n}\text{ hits }\partial B(x_{N+1}^\delta,\eps)\cup \partial B(x_{N+1}^\delta,\eps)\}
\end{align*}
and
\begin{align*}
&\{d(\gamma_N^{\delta_n},\gamma_N)\le\eps\}\cap\left(\{\tau<\sigma\}\setminus(\{\gamma_N^{\delta_n}\text{ ends at }(x_{N+1}^{\delta_n}x_{N+2}^{\delta_n})\}\right)\\
\subset &\{\tau_{N,\eps}^{\delta_n}<\infty,\text{ and then }\gamma_N^{\delta_n}\text{ hits }\partial B(\gamma_N^{\delta_n}(\tau_{N,\eps}^{\delta_n}),\sqrt\eps)\text{ before hitting }(x_{N+1}^{\delta_n}x_{N+2}^{\delta_n})\}
\\
&\cup\{\gamma_N^{\delta_n}\text{ hits }\partial B(x_{N+1}^\delta,\eps)\cup \partial B(x_{N+1}^\delta,\eps)\}.
\end{align*}
Combining with Lemma~\ref{lem::conv_aux1}, we have that 
\begin{equation}\label{eqn::diff_aux}
\mathcal Q\left[\one_{\{d(\gamma_N^{\delta_n},\gamma_N)\le\eps\}}\left|\one_{\{\gamma_N^{\delta_n}\text{ ends at }(x_{N+1}^\delta x_{N+2}^{\delta_n})\}}-\one_{\{\tau<\sigma\}}\right|\right]\le C_1\eps^{C_2}.
\end{equation}
By choosing $n$ large enough, we have that
\begin{equation}\label{eqn::dis_aux}
\mathcal Q[d(\gamma_N^{\delta_n},\gamma_N)\ge\eps]<\eps.
\end{equation}
Combining~\eqref{eqn::diff_aux} and~\eqref{eqn::dis_aux} together, we have that $\one_{\{\gamma_N^{\delta_n}\text{ ends at }(x_{N+1}^\delta x_{N+2}^\delta)\}}$ converges to $\one_{\{\tau<\sigma\}}$ in probability.  By almost the same argument, we can prove that $\gamma_j$ hits $\partial\Omega$ only at its two ends for $1\le j\le N$. This completes the proof.
\end{proof}
Before the proof of Theorem~\ref{thm::Ising_limit}, we derive some useful properties of functions $\LR_N$ and $\mathcal Z_N$.
\begin{lemma}\label{lem::mart_gammaN}
FIx $N\ge 1$. The function $\LR_N$ satisfies the PDE systems: for $1\le i\le N$, we have that ($\kappa=3$ and $h=\frac{1}{2}$)
\begin{equation}\label{eqn::BPZ}
\left(\frac{\kappa}{2}\partial_{ii}^2+\sum_{\substack{1\le j\le N+1\\j\neq i}}\frac{2}{y_j-y_i}\partial_j-\sum_{\substack{1\le j\le N\\j\neq i}}\frac{2h}{(y_j-y_i)^2}-\frac{1}{8}\frac{1}{(y_{N+1}-y_i)^2}\right)R=0
\end{equation}
Moreover, let $y_j:=\phi(x_j)$ in Lemma~\ref{lem::conv_aux} for $1\le j\le N$. Suppose $\ell$ is $\SLE_\kappa\left(2,\ldots,2;\frac{\kappa-6}{2}\right)$  from $y_N$ to $\infty$ with force points $x_1^L=y_{N-1},\ldots,x_{N-1}^L=y_1$ on its left and $x_1^R=y_{N+1}$ on its right. Denote by $(W^{\ell}_t:t\ge 0)$ the driving function of $\ell$ and by $(g_t^\ell:T\ge 0)$ the corresponding conformal maps. Define 
\begin{align*}
M_t:=&\frac{\prod_{1\le i<j\le N-1}(g_t^\ell(y_j)-g_t^\ell(y_i))^{-\frac{2}{3}}\prod_{1\le i\le N-1}(g_t^\ell(y_{N+1})-g_t^\ell(y_i))^{\frac{1}{2}}}{\prod_{1\le i<j\le N-1}(y_j-y_i)^{-\frac{2}{3}}\prod_{1\le i\le N-1}(y_{N+1}-y_i)^{\frac{1}{2}}}\\
&\times \frac{(g_t^\ell(y_{N+1})-W_t^\ell)^{\frac{1}{2}}\times\prod_{1\le i\le N-1}(W_t^\ell-g_t^\ell(y_i))^{-\frac{2}{3}}}{(y_{N+1}-y_N)^{\frac{1}{2}}\times\prod_{1\le i\le N-1}(y_N-y_i)^{-\frac{2}{3}}}\times\frac{\LR_N(g_t^\ell(y_1),\ldots, g_t^\ell(y_{N-1}), W_t^\ell; g_t^\ell(y_{N+1}))}{\LR_N(y_1,\ldots,y_N;y_{N+1})}.
\end{align*}
Then, we have that $\{M_t\}_{t\ge 0}$ is a local martingale for $\ell$. For $\eps>0$ and $K>0$, define $\tau_\eps$ to be the hitting time of $\ell$ at the union of $\eps$-neighbourhood of $(y_{N+1},+\infty)$ and $\partial B(y_N,\frac{1}{\eps})$, define $L_K$ to be the first time that $M_t\ge K$. Then,  the law of $\ell$ weighted by $\left\{M_{t\wedge\tau_\eps\wedge L_K}\right\}_{t\ge 0}$ equals the law of $\gamma_N$ in Lemma~\ref{lem::conv_aux} stopped at corresponding stopping time.
\end{lemma}
\begin{proof}
The proof follows from the commutation relations for $\SLE$s~\cite{DubedatCommutationSLE}. 
We  first prove that  curves $(\gamma_1,\ldots,\gamma_N)\sim\mathcal Q$ satisfy the following commutation relation: Take any disjoint neighborhoods $U_1,\ldots, U_N$ such that $U_j$ contains $x_j$ on its boundary and $U_j\cap\left(\{x_1,\ldots,x_{j-1},x_{j+1},\ldots, x_{N}\}\cup [x_{N+1}x_{N+2}]\right)=\emptyset$ for $1\le j\le N$. Denote by $T_j$ the hitting time of $\gamma_j$ at $\overline{\partial U_j\cap\Omega}$. Denote by $\phi_T$ the conformal map from $\Omega_T:=\Omega\setminus\cup_{1\le i\le N}\gamma_i[0,T_i]$ onto $\HH$ such that $\phi_T(x_{N+2})=\infty$ and $\phi_T(\gamma_1(T_1))<\ldots<\phi_T(\gamma_N(T_N))<\phi_T(x_{N+1})$. Denote by $\tau_j$ the hitting time of $\gamma_j$ at $\cup_{i\neq j}\gamma_i[0,T_i]\cup (x_{N+1}x_{N+2})$. Then, the driving function of $\phi_T(\gamma_j[T_j,\tau_j))$ satisfies SDE~\eqref{eqn::local_driving} if we  replace $(\phi(x_1),\ldots,\phi(x_{N+1}))$ by $(\phi_T(\gamma_1(T_1)),\ldots,\phi_T(x_{N+1}))$.

Choose stopping times $T_j^\delta$ for $\gamma_j^{\delta_n}$ such that $T_j^\delta$ converges to $T_j$ almost surely under $\mathcal Q$. Then, we have that $\Omega_{T}^{\delta_n}:=\Omega_{\delta_n}\setminus\cup_{1\le i\le N}\gamma_i^{\delta_n}[0,T_i^{\delta_n}]$ converges to $\Omega_T$ in the Carath\'eodory sense almost surely. Choose conformal map $\phi^{\delta_n}_T$ from $\Omega_T^{\delta_n}$ onto $\HH$, such that $\phi^{\delta_n}_T(x^{\delta_n}_{N+2})=\infty$ and $\phi^{\delta_n}_T(\gamma^{\delta_n}_1(T^{\delta_n}_1))<\ldots<\phi^{\delta_n}_T(\gamma^{\delta_n}_N(T^{\delta_n}_N))<\phi^{\delta_n}_T(x^{\delta_n}_{N+1})$ and $\phi^{\delta_n}_T$ converges to $\phi_T$ locally uniformly. Denote by $\tau^{\delta_n}_j$ the hitting time of $\gamma^{\delta_n}_j$ at $\cup_{i\neq j}\gamma^{\delta_n}_i[0,T^{\delta_n}_i]\cup[x^{\delta_n}_{N+1}x^{\delta_n}_{N+2}]$.   
\begin{itemize}
\item
By~\cite[Theorem 3.1]{IzyurovObservableFree}, we have that the driving function of $\phi_T^{\delta_n}(\gamma_j^{\delta_n}[T_j^{\delta_n},\tau_j^{\delta_n}))$ converges the solution of SDE~\eqref{eqn::local_driving} in law if we  replace $(\phi(x_1),\ldots,\phi(x_{N+1}))$ by $(\phi_T(\gamma_1(T_1)),\ldots,\phi_T(x_{N+1}))$. 
\item
By the convergence of $\gamma_i^{\delta_n}$ for $1\le i\le N$, and the fact that the limit $\gamma_i$ is simple (due to Lemma~\ref{lem::conv_aux} and Lemma~\ref{lem::conv_aux2}), we have that the limit of the driving function of $\phi_T^{\delta_n}(\gamma_j^{\delta_n}[T_j^{\delta_n},\tau_j^{\delta_n}))$ equals the driving function of $\phi(\gamma_j[T_j,\tau_j))$ almost surely.
\end{itemize}
Combining these two facts together, we have that the driving function of $\phi(\gamma_j[T_j,\tau_j))$ satisfies the required SDE (See also~\cite[Remark 3.2]{IzyurovObservableFree}).
By~\cite[Theorem 7]{DubedatCommutationSLE}, there exists a constant $\mu\in\R$, such that on $\Xi_{N+1}$, we have
\begin{equation}\label{eqn::BPZ_aux}
\left(\frac{\kappa}{2}\partial_{ii}^2+\sum_{\substack{1\le j\le N+1\\j\neq i}}\frac{2}{y_j-y_i}\partial_j-\sum_{\substack{1\le j\le N\\j\neq i}}\frac{2h}{(y_j-y_i)^2}-\frac{2\mu}{(y_{N+1}-y_i)^2}\right)R=0
\end{equation}
By definition of $\LR_N$, we have 
\[\lim_{y_N\to y_{N+1}}(y_{N+1}-y_N)^2\partial^2_{NN}\LR_N(y_1,\ldots,y_N;y_{N+1})=\frac{3}{4}\LR_{N-1}(y_1,\ldots,y_{N-1};y_N)\]
and
\[\lim_{y_N\to y_{N+1}}(y_{N+1}-y_N)\partial_{N+1}\LR_N(y_1,\ldots,y_N;y_{N+1})=-\frac{1}{2}\LR_{N-1}(y_1,\ldots,y_{N-1};y_N).\]
By~\cite[Remark 2.5]{IzyurovObservableFree}, we have that $\LR_{N-1}(y_1,\ldots,y_{N-1};y_N)\neq 0$. Plugging into~\eqref{eqn::BPZ_aux}, we obtain $\mu=\frac{1}{16}$. This proves~\eqref{eqn::BPZ}.

By I\'to's formula, we can conclude that $\{M_t\}_{t\ge 0}$ is a local martingale for $\ell$. The last conclusion is due to the Girsanov's theorem. This completes the proof.
\end{proof}
\begin{lemma}\label{lem::bound_control}
For $K>0$, define
\[\Xi_{N+1,K}:=\left\{(y_1,\ldots,y_{N+1})\in\Xi_{N+1}: |y_i|\le K\text{ for }1\le i\le N+1\text{ and }|y_{i}-y_{i-1}|\ge\frac{1}{K}\text{ for }2\le i\le N\right\}.\]
Then, we have the following properties.
\begin{itemize}
\item
We have that $\LR_N>0$ on $\Xi_{N+1}$.
\item
 There exists a constant $C=C(K)>0$ such that
\[\sup_{(y_1,\ldots,y_{N+1})\in\Xi_{N+1,K}}\frac{\mathcal Z_N(y_1,\ldots,y_N;y_{N+1})}{\prod_{1\le i<j\le N}(y_j-y_i)^{-\frac{2}{3}}\prod_{1\le i\le N}(y_{N+1}-y_i)^{\frac{1}{2}}\times\LR_N(y_1,\ldots,y_N;y_{N+1})}\le C(K).\]
\end{itemize}
\end{lemma}
\begin{proof}
For the first item, we can prove by induction. Suppose it holds when $N\le n-1$, when $N=n$, we have
\[\lim_{n\to\infty}(y_{n}-y_{n-1})^{\frac{1}{2}}\LR_n(y_1,\ldots,y_{n-1};y_{n})=\LR_{n-1}( y_1,\ldots,\hat y_{n-2};\hat y_{n-1})>0.\]
By~\cite[Remark 2.5]{IzyurovObservableFree}, we have that $\LR_n$ is non-zero on $\Xi_{N+1}$. This implies that $\LR_n$ is positive on $\Xi_{N+1}$ and completes the proof of the first item.

For the second item, note that $\prod_{1\le i<j\le N}(y_j-y_i)^{-\frac{2}{3}}\prod_{1\le i\le N-1}(y_{N+1}-y_i)^{\frac{1}{2}}$ is uniformly bounded from above on $\Xi_{N+1,K}$. Similar to~\eqref{eqn::bound_U_1}, we have that  $\mathcal Z_N$ is uniformly bounded from above on $\Xi_{N+1,K}$.
Suppose $\{(y^n_1,\ldots,y^n_{N+1})\}_{n\ge 1}$ is a sequence of points such that 
\[
\lim_{n\to\infty}(y_{N+1}-y_N)^{\frac{1}{2}}\LR_N(y^n_1,\ldots,y^n_N;y^n_{N+1})=\inf_{(y_1,\ldots,y_{N+1})\in\Xi_{N+1,K}}(y_{N+1}-y_N)^{\frac{1}{2}}\LR_N(y_1,\ldots,y_N;y_{N+1}).
\]
We may assume that $\lim_{n\to\infty}y_j^n=\hat y_j$ for $1\le j\le N+1$.
\begin{itemize}
\item
If $\hat y_N\neq \hat y_{N+1}$, we have
\[\lim_{n\to\infty}(y_{N+1}-y_N)^{\frac{1}{2}}\LR_N(y^n_1,\ldots,y^n_N;y^n_{N+1})=(\hat y_{N+1}-\hat y_N)^{\frac{1}{2}}\LR_N(\hat y_1,\ldots,\hat y_N;\hat y_{N+1}),\]
which is nonzero by~\cite[Remark 2.5]{IzyurovObservableFree}.
\item
If $\hat y_N=\hat y_{N+1}$, we have
\[\lim_{n\to\infty}(y_{N+1}-y_N)^{\frac{1}{2}}\LR_N(y^n_1,\ldots,y^n_N;y^n_{N+1})=\LR_{N-1}(\hat y_1,\ldots,\hat y_{N-1};\hat y_{N}),\]
which is nonzero by~\cite[Remark 2.5]{IzyurovObservableFree}.
\end{itemize}
This completes the proof.
\end{proof}
Now, we come back to the proof of Theorem~\ref{thm::Ising_limit}. As in Lemma~\ref{lem::mart_gammaN}, denote by $y_j=\phi(x_j)$ for $1\le j\le N$. By Lemma~\ref{lem::curve_properties}, the second line and the third line in~\eqref{eqn::limt_prob} are equal to each other, which we denote by $F_N(y_1,\ldots,y_{N};y_{N+1})$.
\begin{proof}[Proof of Theorem~\ref{thm::Ising_limit}]
We will prove by induction. 

We first consider the case $N=1$. By Lemma~\ref{lem::conv_aux} and Lemma~\ref{lem::conv_aux1}, we have that $\gamma_1^\delta$ converges to $\SLE_3\left(-\frac{3}{2},-\frac{3}{2}\right)$. Due to the explicit form of $\mathcal Z_1,\mathcal W_1$ and $\LR_1$, the second line and the third line in~\eqref{eqn::limt_prob} equals $1$. Thus, Theorem~\ref{thm::Ising_limit} holds when $N=1$.

Next, suppose that Theorem~\ref{thm::Ising_limit} holds for $N\le m-1$. By tightness, It suffices to prove Theorem~\ref{thm::Ising_limit} when $N=m$ along a subsequence. By Lemma~\ref{lem::conv_aux2}, we can choose a subsequence, which we still denote by $\{\delta_n\}_{n\ge 1}$, such that $\one_{\{\gamma_N^{\delta_n}\text{ ends at }(x_{N+1}^\delta,x_{N+2}^\delta)\}}$ converges to $\one_{\{\tau<\sigma\}}$ almost surely under $\mathcal Q$. 

We first prove~\eqref{eqn::limt_prob} and derive the law of $\gamma_N$.
Denote by $\hat\Omega_{\delta_n}$ (resp. $\hat\Omega$) the connected component of $\Omega_{\delta}\setminus\gamma_N^{\delta_n}$ (resp. $\Omega\setminus\gamma_N$) which contains $x_1^{\delta_n}$ (resp. $x_1$). Then, we have that $\hat\Omega_{\delta_n}$ converges to $\hat\Omega$ in the close-Carath\'eodory sense. Denote by $\tau_{\delta_n}$ the hitting time of $\gamma_N^{\delta_n}$ at $(x_{N+1}^{\delta_n}x_{N+2}^{\delta_n})$. Suppose $f$ is a continuous and bounded function on curve space, by domain Markov property, we have
\[\E[\one_{A_{\delta_n}}f(\gamma_N^{\delta_n})]=\E\left[\E[A_{\delta_n}(\hat\Omega_{\delta_n};x_1^{\delta_n},\ldots,x_{N-1}^{\delta_n};\gamma_N^{\delta_n}(\tau_{\delta_n}),x_{N+2}^{\delta_n})]f(\gamma_N^{\delta_n})\one_{\{\gamma_N^{\delta_n}\text{ ends at }(x_{N+1}^\delta x_{N+2}^\delta)\}}\right]\]
 Denote by $(W_t:t\ge 0)$ the driving function of $\phi(\gamma_N)$ and by $(g_t:t\ge 0)$ the corresponding conformal maps. By induction hypothesis, we have that 
\begin{equation}\label{eqn::final_aux1}
\lim_{n\to\infty}\E[\one_{A_{\delta_n}}f(\gamma_N^{\delta_n})]=\mathcal Q\left[F_{N-1}(g_\tau(y_1),\ldots, g_\tau(y_{N-1}; W_\tau))f(\gamma_N)\one_{\{\tau<\sigma\}}\right].
\end{equation}
By Lemma~\ref{lem::cascade_eta} and Lemma~\ref{lem::mart_gammaN}, we have that $\{N_t:=F_N(g_t(y_1),\ldots, g_t(y_{N-1}),W_t; g_t(y_{N+1}))\}_{t\ge 0}$ is a local martingale for $\gamma_N$. Choose $K$ large enough. Recall the definition of $\Xi_{N+1,K}$ in Lemma~\ref{lem::bound_control}. Define $S_K$ the exiting time of $(g_t(y_1),\ldots, g_t(y_{N-1}),W_t; g_t(y_{N+1}))$ at $\Xi_{N+1,K}$. Denote by $\tau_\eps$ the hitting time of $\phi(\gamma_N)$ at the union of $\eps$-neighbourhood of $(y_{N+1},+\infty)$ and $\partial B(y_N,\frac{1}{\eps})$ and by $\sigma_\eps$ the hitting time of $\phi(\gamma_N)$ at $\eps$-neighbourhood of $\{y_1,\ldots,y_{N-1}\}$.
By Lemma~\ref{lem::bound_control}, there exists a constant $C(K)>0$ such that for every $t\ge 0$, we have 
\[N_{t\wedge\tau_\eps\wedge\sigma_\eps\wedge S_K}\le C(K).\]
Moreover, weighted by $\left\{\frac{N_{t\wedge\tau_\eps\wedge\sigma_\eps\wedge S_K}}{N_0}\right\}_{t\ge 0}$, the law of $\gamma_N[0,\tau_\eps\wedge\sigma_\eps\wedge S_K]$ equals the law of $\ell_N[0,\tau^{\ell_N}_\eps\wedge\sigma^{\ell_N}_\eps\wedge S^{\ell_N}_K]$ (See Theorem~\ref{thm::multiple_odd} for the definition of $\ell_N$), where we denote by $\tau^{\ell_N}_\eps,\sigma^{\ell_N}_\eps,S^{\ell_N}_K$ for $\ell_N$ similarly to $\tau_\eps, \sigma_\eps, S_K$. Thus, we have
\[
\mathcal Q[N_{\tau_\eps\wedge\sigma_\eps\wedge S_K}\one_{\{\sigma_\eps\wedge S_K<\tau_\eps\}}]=N_0\PP[\sigma^{\ell_N}_\eps\wedge S^{\ell_N}_K<\tau^{\ell_N}_\eps]
\]
In particular, since $\ell_N$ is a simple curve which hits $(y_{N+1},+\infty)$ before hitting $(-\infty, y_N)$, we have
\begin{equation}\label{eqn::mart_gammaN_aux1}
\lim_{\eps\to 0, K\to\infty}\mathcal Q[N_{\tau_\eps\wedge\sigma_\eps\wedge S_K}\one_{\{\sigma_\eps\wedge S_K<\tau_\eps\}}]\le N_0\PP[\ell_N\text{ hits }\{y_1,\ldots, y_{N-1}\}]=0.
\end{equation}
On the event $\{\tau<\sigma\wedge S_K\}$, we have that 
\begin{align}\label{eqn::mart_gammaN_aux2}
&\lim_{\eps\to 0}N_{\tau_\eps}\notag\\
=&\frac{1}{B\left(\frac{2}{\kappa},\frac{2}{\kappa}\right)^{\left[\frac{N}{2}\right]}}\prod_{1\le i<j\le N-1}(g_\tau(y_j)-g_\tau(y_i))^{\frac{2}{3}}\prod_{1\le i\le N-1}(W_\tau-g_\tau(x_i))^{-\frac{1}{2}}\times\frac{\mathcal W_{N-1}(g_\tau(y_1),\ldots, g_\tau(y_{N-1});W_\tau)}{\LR_{N-1}(g_\tau(y_1),\ldots,g_\tau(y_{N-1});W-\tau)}\notag\\
=&F_{N-1}(g_\tau(y_1),\ldots, g_\tau(y_{N-1}; W_\tau)).
\end{align}
Denote the limit by $N_\tau$. Combining~\eqref{eqn::mart_gammaN_aux1} and~\eqref{eqn::mart_gammaN_aux2} together,  by letting $\eps\to 0$,
and $K\to \infty$, by dominated convergence theorem, we have
\begin{equation}\label{eqn::mart_gammaN_aux4}
\mathcal Q[f(\gamma_N[0,\tau])N_{\tau}\one_{\{\tau<\sigma\}}]=N_0\E\left[f(\ell_N[0,\tau^{\ell}_N])\right].
\end{equation}
By taking $f=1$ in~\eqref{eqn::final_aux1}, combining with~\eqref{eqn::mart_gammaN_aux4}, we have
\[\lim_{n\to \infty}\PP[A_{\delta_n}]=N_0.\]
This gives~\eqref{eqn::limt_prob}. Thus,~\eqref{eqn::final_aux1} and~\eqref{eqn::mart_gammaN_aux4} implies that
\begin{equation}\label{eqn::mart_gammaN_aux5}
\lim_{n\to\infty}\E[f(\gamma_N^{\delta_n})\cond A_{\delta_n}]=\E\left[f(\ell_N[0,\tau^{\ell}_N])\right].
\end{equation}
Thus, the conditional law of $\gamma_N^{\delta_n}$ given $A_{\delta_n}$ converges to $\ell_N$.

Next, by induction hypothesis, we have that the conditional law of $(\gamma_1^{\delta_n},\ldots,\gamma_{N-1}^{\delta_n})$ converges to multiple $N-1$-$\SLE_3\left(-\frac{3}{2},-\frac{3}{2}\right)$ on $X_{N-1}(\hat\Omega; x_1,\ldots,x_{N-1}; \gamma_N(\tau),x_{N+2})$. Combining with Theorem~\ref{thm::multiple_odd}, Theorem~\ref{thm::uniqueness_multiple} and~\eqref{eqn::weight_ell} and~\eqref{eqn::mart_gammaN_aux5}, we have that the random curves $(\gamma_1,\ldots,\gamma_N)$ is the multiple $N$-$\SLE_3\left(-\frac{3}{2},-\frac{3}{2}\right)$ on $X_{N}(\Omega; x_1,\ldots,x_{N}; x_{N+1},x_{N+2})$.

This completes the induction and hence completes the proof.
\end{proof}

\newcommand{\etalchar}[1]{$^{#1}$}

\end{document}